\documentclass[a4paper,11pt]{article}
\usepackage{xypic}
\usepackage{mathtools}
\usepackage{amsmath}
\usepackage{mathrsfs}
\usepackage{amssymb}
\usepackage{amsthm}
\usepackage{dsfont}
\usepackage{graphicx}
\usepackage{bmpsize}
\usepackage{times}
\usepackage{ textcomp }
\usepackage{color}
\usepackage{mathtools}
\usepackage{latexsym, empheq, fancybox}
\usepackage{mathrsfs}
\usepackage{exscale}
\usepackage{indentfirst}
\usepackage{booktabs, array}
\usepackage[english]{babel}
\numberwithin{equation}{section}
\newtheorem{theorem}{Theorem}[section]
\newtheorem{lemma}[theorem]{Lemma}
\newtheorem{definition}[theorem]{Definition}

\newtheorem{corollary}[theorem]{Corollary}

\newtheorem{remark}[theorem]{Remark}
\newtheorem{proposition}[theorem]{Proposition}
\newtheorem{question}[theorem]{Question}
\usepackage{authblk}
\title{\Large\textbf{Group Extensions for Random Shifts of Finite Type}}

\author{
		Kexiang Yang
		\thanks{School of Mathematical Sciences and Shanghai Center for Mathematical Sciences, Fudan University, Shanghai 200433,
                P. R. China
			(E-mail: kxyangs@163.com)},
		Ercai Chen
		\thanks{School of Mathematical Sciences and Institute of Mathematics,Ministry of Education Key Laboratory of NSLSCS, Nanjing Normal University, Nanjing 210023, P. R. China,
			and Center  of Nonlinear Science, Nanjing University, Nanjing 210093, P. R. China
			(E-mail: ecchen@njnu.edu.cn)
		},
		Zijie Lin
		\thanks{School of Mathematics, University of Science and Technology of China, Hefei, Anhui, 230026, P. R. China
			(E-mail: zjlin137@126.com)},
		Xiaoyao Zhou
		\thanks{School of Mathematical Sciences and Institute of Mathematics, Ministry of Education Key Laboratory of NSLSCS, Nanjing Normal University, Nanjing 210023, P. R. China
			(E-mail: zhouxiaoyaodeyouxian@126.com)}
}

\begin{document}
\date{}

\maketitle
\begin{abstract}
Symbolic dynamical theory plays an important role in the research of amenability with a countable group. Motivated by the deep results of Dougall and Sharp, we study the group extensions for topologically mixing random shifts of finite type. For a countable group $G$, we consider the potential connections between relative Gurevi\v{c} pressure (entropy), the spectral radius of random Perron-Frobenius operator and amenability of $G$. Given $G^{\rm ab}$ by the abelianization of $G$ where $G^{\rm ab}=G/[G,G]$, we consider the random group extensions of random shifts of finite type between $G$ and $G^{\rm ab}$. It can be proved that the relative Gurevi\v{c} entropy of random group $G$ extensions is equal to the relative Gurevi\v{c} entropy of random group $G^{\rm ab}$ extensions if and only if $G$ is amenable. Moreover, we establish the relativized variational principle and discuss the unique equilibrium state for random group $\mathbb{Z}^{d}$ extensions.
\end{abstract}


\section{Introduction}

The concept of group extensions, introduced into the research of symbolic dynamics,
has constituted an essential component in the description of the
amenability of countable discrete groups from probability and geometry theory.
In 1954, F{\o}lner \cite{Folner1954Generalization} proved that a necessary and sufficient condition that there exists a real linear functional $\mathfrak{M}$ on a space $L$ of bounded real functions on a countable group $G$ with the properties: (1) $\inf_{x\in G}h(x)\leq \mathfrak{M}\{x\rightarrow h(x)\}\leq \sup_{x\in G}h(x)$; (2) $\mathfrak{M}\{x\rightarrow h(xa)\}=\mathfrak{M}\{x\rightarrow h(x)\}$ for any $a\in G$ and $h\in L$ is that $$\sup_{x\in G}H(x)\geq 0$$ for any function $H$ of the form $$H(x)=\sum_{i=1}^{n}(h_{i}(x)-h_{i}(xa_{i}))$$ where $i=1,\ldots,n$, $h_{i}\in L$ and $a_{i}\in G$. The linear functional $\mathfrak{M}$ is called a right-invariant \emph{Banach mean value} on $L$. Moreover, if for any $a,b\in G$, the linear functional $\mathfrak{M}$ satisfies
\begin{align*}
\mathfrak{M}\{x\rightarrow h(axb)\}=\mathfrak{M}\{x\rightarrow h(x)\},
\end{align*}
we call that $\mathfrak{M}$ is a bi-invariant Banach mean value.
Note that if there exists a right-invariant Banach mean value on $L$, there also exists a bi-invariant Banach mean value. It is called that $G$ has a \emph{full Banach mean value} if there exists a bi-invariant Banach mean value.  And then F{\o}lner \cite{Folner1955On} studied the countable groups equipped with a full Banach mean value and established a F{\o}lner theorem: the countable group $G$ has a full Banach mean value if and only if for any $0<k<1$, and any finite elements $a_{1},a_{2},\ldots, a_{n}\in G$, there exists a finite subset $E$ of $G$ such that
\begin{align*}
N(E\cap Ea_{i})\geq k\cdot N(E)
\end{align*}
where $N(E)$ denotes the number of elements in $E$.
We say that $G$ is \emph{amenable} if $G$ has a full Banach mean value. Later in 1959, Kesten \cite{Kesten1959Full} studied the amenability of $G$ by the spectral radius of the a Markov operator $M$ on $l^{2}(G)$ with a symmetric probability distribution $p$ on $G$, that is, $p(x)\geq 0$, $p(x)=p(x^{-1})$ for any $x\in G$ and $\sum_{x\in G}p(x)=1$, where $M$ is given by
\begin{align*}
Mh(x):=\sum_{y\in G}p\left(x\cdot y^{-1}\right)\cdot h(y)
\end{align*}
for any $h\in l^{2}(G)$.
For a probabilistic method, Kesten \cite{Kesten1959Full} proved that the spectral radius of the Markov operator $M$ on $l^{2}(G)$ is equal to $1$ if and only if $G$ is amenable. For a general case, Day \cite{Day1964Convolutions} considered a criterion for amenability of $G$ associated with a convolution operator on $l^{2}(G)$. The other discussions in terms of groups were referred to \cite{Brooks1985The,Lindenstrauss2001Pointwise,Roblin2005A,Sharp2007Critical,Dougall2016Amenability,Dougall2019Critical,Downarowicz2023Symbolic,Downarowicz2023Multiorders}. A question appears whether there is an equivalent characterization for amenability of $G$ by using the method of dynamical systems. This topic was discussed in the relations between amenability and topological dynamical systems by \cite{Sharp2007Critical,Stadlbauer2013An,Jaerisch2014Fractal,Jaerisch2016Recurrence,Dougall2021Anosov}, and symbolic dynamics with respect to countable alphabet has become a powerful technical tool in the study of amenability of $G$. Let us turn back to the background of the above topic. In 1969 and 1970, Gurevi\v{c} \cite{Gurevic1969Topological,Gurevic1970Shift} introduced an entropy for a topological countable Markov shift, which is called the \emph{Gurevi\v{c} entropy}. And then Sarig \cite{Sarig1999Thermodynamic} defined the definition of \emph{Gurevi\v{c} pressure} for a topological countable Markov shift, which extended the notion of Gurevi\v{c} entropy. The thermodynamic formalism and topological countable Markov shifts were also studied by \cite{Sarig1999Thermodynamic,Sarig2001Thermodynamic,Sarig2003Existence,Sarig2013Symbolic}. When it comes to the amenability of $G$, Stadlbauer \cite{Stadlbauer2013An} considered group extensions of topological countable Markov shifts with topologically mixing property, established that (1) if the group extension is a topologically transitive, symmetric group extension of the topologically mixing topological countable Markov shift, the potential function is H\"{o}lder continuous and weakly symmetric, and the Gurevi\v{c} pressure is finite, then the Gurevi\v{c} pressure of the group extension is equal to the Gurevi\v{c} pressure of topological countable Markov shift if $G$ is amenable; (2) assume that the group extension is a topologically transitive group extension of the topologically mixing topological countable Markov shift equipped with the BIP property, if the Gurevi\v{c} pressure of the group extension is equal to the Gurevi\v{c} pressure of topological countable Markov shift, then the countable group $G$ is amenable. Besides, the spectral radius of the Perron-Frobenius operator can be used to characterise amenability of $G$. In 2015, Jaerisch \cite{Jaerisch2015Groupextended} proved that the logarithm of the spectral radius of the Perron-Frobenius operator coincides with the Gurevi\v{c} pressure of topological countable markov shifts if and only if $G$ is amenable. But when $G$ is amenable, it may happen that the Gurevi\v{c} pressure of the group extension is not equal to the Gurevi\v{c} pressure of topological countable Markov shift (such an example for $G=\mathbb{Z}$ in \cite{Dougall2021Anosov}).
Let $M$ be a compact smooth Riemannian manifold and $\phi:M\rightarrow M$ be a transitive Anosov flow.
Now suppose that $X$ is a regular cover of $M$ with covering group $G$, i.e. $G$
acts freely and isometrically on $X$ such that $M = X/G$.
The interest is that the same question is derived form growth rates of periodic orbit of an Anosov flow $\phi$ and a lifted flow $\phi_{X}$, that is, $h_{top}(\phi)$ and $h_{Gur}(\phi_{X})$, respectively. In a special case, the above equivalence holds. Let $M = SV$ be the unit-tangent
bundle over a compact manifold $V$ with negative sectional curvatures and
$\phi_{X}$ be the geodesic flow, by \cite{Dougall2016Amenability} and \cite{Roblin2005A}, then
$h_{Gur}(\phi_{X})=h_{top}(\phi)$ if and only if $G$ is amenable. This equivalence fails for general Anosov flows. So the purpose is to find a proper entropy to compare $h_{Gur}(\phi_{X})$.
Later in 2021, Dougall and Sharp \cite{Dougall2021Anosov} studied the growth properties of group extensions of hyperbolic dynamical systems where they do not assume the symmetric condition. For the abelianization $$G^{\rm ab}=G/[G,G]$$ of a countable group $G$, Dougall and Sharp \cite{Dougall2021Anosov} established that the Gurevi\v{c} entropy of the group $G$ extension coincides with the Gurevi\v{c} entropy of the group $G^{\rm ab}$ extension for topological countable Markov shifts is and only if $G$ is amenable. When considering random Anosov flows, it seems that there are no periodic orbits in the usual sense. Our starting point is the theory of relative thermodynamic formalism. When it involves the systems whose evolution can be represented by compositions of different transformations, people will study symbolic dynamics under the actions of random dynamical systems, which is called random shifts of finite type. A question arises naturally whether this characterization admit a random countable Markov shift.
This motivates us to characterize the relations between amenability and random dynamical systems.
Let $(\Omega,\mathcal{F},\mathbb{P})$ be a probability space and $\vartheta:\Omega\rightarrow\Omega$ be a $\mathbb{P}$-preserving ergodic invertible transformation. Let $(X,\mathcal{B})$ be a metric space together with the distance function $d$ and the Borel $\sigma$-algebra $\mathcal{B}$. Let $\mathcal{E}\subset \Omega\times X$ be a measurable with respect to the product $\sigma$-algebra $\mathcal{F}\times \mathcal{B}$ and such that the fibers $$\mathcal{E}_{\omega}=\{x\in X:(\omega,x)\in \mathcal{E}\}, \ \omega\in \Omega,$$ are non-empty. A \emph{continuous bundle random dynamical system (RDS)} $f$ over $(\Omega,\mathcal{F},\mathbb{P},\vartheta)$ is generated by mappings $f_{\omega}:\mathcal{E}_{\omega}\rightarrow \mathcal{E}_{\vartheta\omega}$, that is,
\begin{align*}
f_{\omega}^{n}=\left\{
\begin{aligned}
f_{\vartheta^{n-1}\omega}\circ f_{\vartheta^{n-2}\omega}\circ\cdots \circ f_{\vartheta\omega}\circ f_{\omega} \ \ \ \ \ \ \ \ \ \ \ &\text{for} \ n>0,\\
Id \ \ \ \ \ \ \ \ \ \ \ \ \ \ \ \ \ \ \ \ \ \ \ \ \ \ \ \ \ \ \ \ \ \ \ \ \ \ \ \ \ \ \ \ \ \ \ \ \ \ \ \ \ \ \ \ \ \ \ &\text{for} \ n=0,
\end{aligned}
\right.
\end{align*}
so that the map $(\omega,x)\mapsto f_{\omega}x$ is measurable and the map $x\mapsto f_{\omega}x$ is continuous for $\mathbb{P}$-a.e. $\omega\in \Omega$. The map
\begin{align*}
\Theta:\mathcal{E}\rightarrow\mathcal{E}, \ \ \Theta(\omega,x)=(\vartheta\omega,f_{\omega}x)
\end{align*}
is called the skew product transformation. Many of results has been established in random dynamical systems (see \cite{Kifer1986Ergodic,Arnold1998Random,Khanin1996Thermodynamic,Dooley2015Local,Huang2017Entropy,Huang2019Ergodic})
The framework of relative thermodynamic formalism has been developed in \cite{Kifer1992Equilibrium,Bogenschutz1995Ruelle,Stadlbauer2021Thermodynamic} for random dynamical systems on random sets with compact fibers. Denker, Kifer and Stadlbauer \cite{Denker2008Thermodynamic} introduced a relative Gurevi\v{c} pressure for random countable Markov shifts equipped with topologically mixing property and studied the relative thermodynamic formalism for it.
It is possible that periodic points do not exist for a random dynamical system. After the works of \cite{Denker2008Thermodynamic}, Stadlbauer \cite{Stadlbauer2010On} gave a relative definition of relative BIP property for random countable Markov shifts and proved that this implies a relative version of the Perron-Frobenius theorem holds. Under the assumption of relative BIP property for random countable Markov shifts, Stadlbauer \cite{Stadlbauer2017Coupling} proved the random version of Ruelle's theorem.
Combining with the above results, group extensions of random countable Markov shifts are good objects to start with.
Firstly, we study the relative Gurevi\v{c} pressure of random countable Markov shifts equipped with topologically mixing property (see Definition \ref{relative Gurevic pressure of random countable Markov shifts}). Then we define the random group extension and give the definition of the relative Gurevi\v{c} pressure for the random group extension equipped with topologically mixing property. Motivated by the results of Dougall and Sharp \cite{Dougall2021Anosov}, suppose that $G^{\rm ab}$ is the abelianization of $G$ where $G^{\rm ab}=G/[G,G]$ and $[G,G]$ is a subgroup of $G$ generated by $$\left\{ghg^{-1}h^{-1}:g,h\in G\right\},$$  the main aim of this paper is to study the amenability of $G$ by consider the relation between the relative Gurevi\v{c} entropy of random group $G$ extensions and the relative Gurevi\v{c} entropy of random group $G^{\rm ab}$ extensions. In the following, we consider the question for a random shift
of finite type:

\begin{question}\label{problem 2 of random group extension and random countable Markov shift}
How one can find a natural comparison of the relative Gurevi\v{c} entropies between random group $G$ and $G^{\rm ab}$ extensions.
\end{question}

When the group $G$ is amenable, we answer the Question \ref{problem 2 of random group extension and random countable Markov shift} and give the following result. For the case of $\varphi=0$, by the Theorem \ref{the third result of main theorems}, we have
\begin{align*}
h_{Gur}^{(r)}(\mathcal{T})=h_{Gur}^{(r)}(\mathcal{T}_{\rm ab}).
\end{align*}

\begin{theorem}\label{the third result of main theorems}
Let $\mathcal{T}$ be a topologically mixing random group extension of a random shift of finite type $f$  by a countable group $G$. Suppose that $\varphi:\mathcal{E}\rightarrow\mathbb{R}$ is a locally fiber H\"{o}lder continuous function satisfying \eqref{the condition of locally fiber Holder continuous function}. If $G$ is amenable, then one has
\begin{align*}
\Pi_{Gur}^{(r)}(\mathcal{T},\tilde{\varphi})=\Pi_{Gur}^{(r)}(\mathcal{T}_{\rm ab},\tilde{\varphi}_{\rm ab}).
\end{align*}
\end{theorem}

For the Question \ref{problem 2 of random group extension and random countable Markov shift}, we prove the converse to Theorem \ref{the third result of main theorems} for the case of $\varphi=0$. For a general case, by the definitions of $h_{Gur}^{(r)}(\mathcal{T}_{\rm ab})$ and $h_{Gur}^{(r)}(\mathcal{T})$, it can be proved that
\begin{align*}
h_{Gur}^{(r)}(\mathcal{T})\leq h_{Gur}^{(r)}(\mathcal{T}_{\rm ab}).
\end{align*}
By using Theorem \ref{the one hand of the fourth result of main theorems}, if $G$ is non-amenable, then we have
$$h_{Gur}^{(r)}(\mathcal{T})<h_{Gur}^{(r)}(\mathcal{T}_{\rm ab}).$$
Combining with Theorem \ref{the third result of main theorems} and Theorem \ref{the one hand of the fourth result of main theorems}, we prove the following main result.

\begin{theorem}\label{the fourth result of main theorems}
Let $\mathcal{T}$ be a topologically mixing random group extension of a random shift of finite type $f$ by a countable group $G$. Then we have $h_{Gur}^{(r)}(\mathcal{T})=h_{Gur}^{(r)}(\mathcal{T}_{\rm ab})$ if and only if $G$ is amenable.
\end{theorem}

Denote by $\mathcal{M}_{\mathbb{P}}^{1}(\mathcal{E},f)$ the sets of $\Theta$-invariant probability measures $\mu$ on $\mathcal{E}$ having the marginal $\mathbb{P}$ on $\Omega$. Since $f$ is a topologically mixing random shift of finite type, by \cite{Kifer2001On} it can be obtained that there exists a relativized variational principle for $h_{Gur}^{(r)}(f)$ in the following statement:
\begin{align}\label{classical relative variational principle}
h_{Gur}^{(r)}(f)=h_{top}(f)=\sup_{\mu\in\mathcal{M}_{\mathbb{P}}^{1}(\mathcal{E},f)}\left\{h_{\mu}^{(r)}(f)\right\}.
\end{align}
And then by the statements of \cite{Kifer2008Thermodynamic,Denker2008Thermodynamic}, there exists a unique $\Theta$-invariant fiber (or relative) Gibbs measure $\mu$ which is defined by $d\mu(\omega,x)=d\mu_{\omega}(x)d\mathbb{P}(\omega)$ and a random variable $\lambda:\Omega\rightarrow(0,+\infty)$ satisfying $\int\log\lambda_{\omega}d\mathbb{P}(\omega)=\Pi_{Gur}^{(r)}(f,\varphi)$ such that there exists some random variable $C_{\varphi}:\Omega\rightarrow (0,+\infty)$ satisfying
\begin{align*}
\int_{\Omega} |\log C_{\varphi}(\omega)|d\mathbb{P}(\omega)<\infty
\end{align*}
such that
\begin{align*}
C_{\varphi}(\omega)^{-1}\leq \Lambda_{n}(\omega)\frac{\mu_{\omega}([x_{0},x_{1},\ldots,x_{n-1}]_{\omega})}{\exp\left(S_{n}\varphi(\omega,x)\right)}\leq C_{\varphi}(\omega),
\end{align*}
for any $x\in\mathcal{E}_{\omega}$ and $\mathbb{P}$-a.e. $\omega\in \Omega$, where $\Lambda_{n}(\omega):=\prod_{i=0}^{n-1}\lambda_{\vartheta^{i}\omega}$. It maximises in the variational principle in \eqref{classical relative variational principle}.
Let $G=\mathbb{Z}^{d}$, then $\psi$ is given by for any $(\omega,x)\in\mathcal{E}$, one has $$\psi(\omega,x):=\left(\psi_{1}(\omega,x),\psi_{2}(\omega,x),\ldots,\psi_{d}(\omega,x)\right).$$
Let $\mu\in\mathcal{M}_{\mathbb{P}}^{1}(\mathcal{E},f)$, define
$$\int \psi d\mu:=\left(\int\psi_{1}d\mu,\int\psi_{2}d\mu,\ldots,\int\psi_{d}d\mu\right).$$
By using the same processes with \eqref{variational principle of random group extensions}, we have the following relativized variational principle.

\begin{theorem}
Let $\mathcal{T}$ be a topologically mixing random group extension of a random shift of finite type $f$ by $\mathbb{Z}^{d}$. Then we have
\begin{align*}
h_{Gur}^{(r)}(\mathcal{T})=\sup_{\mu\in\mathcal{M}_{\mathbb{P}}^{1}(\mathcal{E},f)}\left\{h_{\mu}^{(r)}(f): \int \psi d\mu={\bf 0}\right\}
\end{align*}
where ${\bf 0}$ is a $d$-dimensional zero vector.
\end{theorem}

As an consequence, we consider the equilibrium measure and give an application for a topologically mixing random $\mathbb{Z}^{d}$ extension of a random shift of finite type $f$.
\begin{corollary}
Let $\mathcal{T}$ be a topologically mixing random group extension of a random shift of finite type $f$ by $\mathbb{Z}^{d}$. Then there exists an unique $\mu\in\mathcal{M}_{\mathbb{P}}^{1}(\mathcal{E},f)$ maximizing in the above variational principle such that $$\int \psi d\mu={\bf 0}.$$
Moreover, the $\mu$ is an unique determined equilibrium measure for some random continuous function of inner product form $$\left(\xi,  \psi\right):=\sum_{i=1}^{d}c_{i}\cdot \psi_{i}$$ where $c:=(c_{1},c_{2},\ldots,c_{d})\in\mathbb{R}^{d}$.
\end{corollary}

In order to prove the Theorem \ref{the fourth result of main theorems}, we consider the relations between random Perron-Frobenius operator $\log {\rm spr}_{\mathcal{H}}(\mathcal{L}_{\tilde{\varphi}}^{\omega})$ of random group extension and the relative Gurevi\v{c} pressure $\Pi_{Gur}^{(r)}(f,\varphi)$ of random countable markov shifts. Then we consider the following question:

\begin{question}\label{problem of proofing with operator and pressure}
For a general case, it can be proved that $\log {\rm spr}_{\mathcal{H}}(\mathcal{L}_{\tilde{\varphi}}^{\omega})\leq\Pi_{Gur}^{(r)}(f,\varphi)$ for $\mathbb{P}$-a.e. $\omega\in \Omega$. It is natural to ask when the equality holds.
\end{question}

The Question \ref{problem of proofing with operator and pressure} includes the connection between the relative Gurevi\v{c} pressure of a random group extension and the relative Gurevi\v{c} pressure of random countable markov shifts. Then we consider the following question:

\begin{question}\label{problem of random group extension and random countable Markov shift}
By Remark \ref{the remark of random group extension}, it is clear that $\Pi_{Gur}^{(r)}(\mathcal{T},\tilde{\varphi})\leq \Pi_{Gur}^{(r)}(f,\varphi)$ by the definitions and it is natural to ask whether $G$ is amenable when the equality holds.
\end{question}

For the Question \ref{problem of proofing with operator and pressure}, we give the random Perron-Frobenius operator of a random group extension, and study the relation between the logarithm of the spectral radius of this random Perron-Frobenius operator and the relative Gurevi\v{c} pressure of random countable markov shifts, which the logarithm of the spectral radius of this random Perron-Frobenius operator does not coincide with the relative Gurevi\v{c} pressure of random countable markov shifts if $G$ is non-amenable.
And then we answer the Question \ref{problem of random group extension and random countable Markov shift} under the condition of random group extension with topologically mixing property.

\begin{theorem}\label{the first result of main theorems}
Let $\mathcal{T}$ be a topologically mixing random group extension of a random countable Markov shift $f$ with relative BIP property by a countable group $G$. Then we have the following results.
\begin{itemize}
  \item Suppose that $\varphi:\mathcal{E}\rightarrow\mathbb{R}$ is a locally fiber H\"{o}lder continuous function satisfying \eqref{the condition of locally fiber Holder continuous function}. If $G$ is non-amenable, then for $\mathbb{P}$-a.e. $\omega\in \Omega$, one has
$$\log {\rm spr}_{\mathcal{H}}(\mathcal{L}_{\tilde{\varphi}}^{\omega})<\Pi_{Gur}^{(r)}(f,\varphi).$$
  \item Suppose that $\varphi:\mathcal{E}\rightarrow\mathbb{R}$ is a locally fiber H\"{o}lder continuous function satisfying \eqref{the condition of locally fiber Holder continuous function}. Then we have $\Pi_{Gur}^{(r)}(f,\varphi)=\Pi_{Gur}^{(r)}(\mathcal{T},\tilde{\varphi})$ implies that group $G$ is amenable.
\end{itemize}
\end{theorem}

It is noticed that the spectral radius of the random Perron-Frobenius operator on fibers can be used to characterise amenability of $G$. Based on the part II of Theorem \ref{the first result of main theorems}, we answer Question \ref{problem of proofing with operator and pressure}. It can be proved that the logarithm of the spectral radius of this random Perron-Frobenius operator coincides with the relative Gurevi\v{c} pressure of random countable markov shifts almost everywhere if and only if $G$ is amenable.

\begin{theorem}\label{the second result of main theorems}
Let $\mathcal{T}$ be a topologically mixing random group extension of a random countable Markov shift $f$ with relative BIP property by a countable group $G$. Suppose that $\varphi:\mathcal{E}\rightarrow\mathbb{R}$ is a locally fiber H\"{o}lder continuous function with $\Pi_{Gur}^{(r)}(f,\varphi)<\infty$. Then we have
$\log {\rm spr}_{\mathcal{H}}(\mathcal{L}_{\tilde{\varphi}}^{\omega})=\Pi_{Gur}^{(r)}(f,\varphi)$ for $\mathbb{P}$-a.e. $\omega\in \Omega$ if and only if $G$ is amenable.
\end{theorem}

Theorem \ref{the first result of main theorems} and Theorem \ref{the second result of main theorems} also hold in the case of random shifts of finite type. By taking $\varphi=0$, then we have the following statement.

\begin{corollary}
Let $\mathcal{T}$ be a topologically mixing random group extension of a random shift of finite type $f$ by a countable group $G$. Then we have
$\log {\rm spr}_{\mathcal{H}}(\mathcal{L}_{\tilde{0}}^{\omega})=h_{Gur}^{(r)}(f)$ for $\mathbb{P}$-a.e. $\omega\in \Omega$ if and only if $G$ is amenable.
\end{corollary}

The paper is organized as follows. In Section \ref{Preliminaries}, we introduce the notions of random countable Markov shifts, relative Gurevi\v{c} pressure, random group extensions and group extensions of random shifts of finite type. In Section \ref{Random group extensions and its Perron-Frobenius operator}, we give the proof of the part I of Theorem \ref{the first result of main theorems}, that is, Theorem \ref{the part I of the first result}. In Section \ref{Random group extensions by amenable groups}, we give the proof of the part II of Theorem \ref{the first result of main theorems}, that is Theorem \ref{the part 2 of the first theorem}. In Section \ref{Amenability and random Perron-Frobenius operator}, we give the proof of Theorem \ref{the second result of main theorems}. In Section \ref{Relative Gurevic pressure for amenable group extensions}, we give the proof of Theorem \ref{the third result of main theorems}. In Section \ref{Relative Gurevic entropy of random group extensions}, we give the proof of Theorem \ref{the fourth result of main theorems}.

\section{Preliminaries}\label{Preliminaries}

\subsection{Random countable Markov shifts}

Let $(\Omega,\mathcal{F},\mathbb{P})$ be a probability space and $\vartheta:\Omega\rightarrow\Omega$ be a $\mathbb{P}$-preserving ergodic invertible transformation. Our setting consists of a $\mathbb{N}\cup \{\infty\}$-valued random variable $l=l(\omega)>1$, the sets $S(\omega)=\{j\in\mathbb{N}:j<l(\omega)\}$ and measurably depending on $\omega\in \Omega$ matrices $$A_{\omega}=\left(\alpha_{i,j}(\omega),i\in S(\omega),j\in S(\vartheta\omega)\right)$$ with entries $\alpha_{i,j}(\omega)\in \{0,1\}$. Define
\begin{align*}
\mathcal{E}_{\omega}=\left\{x=(x_{0},x_{1},x_{2},\ldots):\alpha_{x_{i},x_{i+1}}(\vartheta^{i}\omega)=1 \ \text{for any} \ i=0,1,2\ldots\right\},
\end{align*}
which we assume to be nonempty for $\mathbb{P}$-a.e. $\omega\in\Omega$, together with the left shifts $f_{\omega}:\mathcal{E}_{\omega}\rightarrow\mathcal{E}_{\vartheta\omega}$ by $(x_{0},x_{1},x_{2},\ldots)\mapsto(x_{1},x_{2},x_{3},\ldots)$. Denote by $\Theta$ the skew product transformation acting on a measurable space $$\mathcal{E}=\{(\omega,x):x\in \mathcal{E}_{\omega},\omega\in \Omega\}$$ by $$\Theta:\mathcal{E}\rightarrow\mathcal{E}, \ \ (\omega,x)\mapsto (\vartheta\omega,f_{\omega}x).$$ Then $f:=(f_{\omega})_{\omega\in\Omega}$ is called a \emph{random countable Markov shift}. For $n\in\mathbb{N}$, we set $f_{\omega}^{n}=f_{\vartheta^{n-1}\omega}\circ f_{\vartheta^{n-2}\omega}\circ\cdots f_{\vartheta\omega}\circ f_{\omega}$ and $\Theta^{n}(\omega,x)=(\vartheta^{n}\omega,f_{\omega}^{n}x)$. A finite word $a=(x_{0},x_{1},\ldots,x_{n-1})\in\mathbb{N}^{n}$ of length $n$ is called $\omega$-admissible if for any $i=0,1,\ldots,n-1$, one has $x_{i}<l(\vartheta^{i}\omega)$ and $\alpha_{x_{i}x_{i+1}}(\vartheta^{i}\omega)=1$. Let $\mathcal{W}^{n}(\omega)$ denote the set of $\omega$-admissible words of length $n$ and $$\mathcal{W}^{\infty}(\omega)=\bigcup_{n\geq 1}\mathcal{W}^{n}(\omega).$$ For any $a\in \mathcal{W}^{\infty}(\omega)$, denote by $|a|$ the length of $a$. In particular, $\mathcal{W}^{1}(\omega)=\{a:a<l(\omega)\}$ and for any $a=(a_{0},a_{1},\ldots,a_{n-1})\in\mathbb{N}^{n}$, the cylinder set is defined by
\begin{align*}
[a]_{\omega}=[a_{0},a_{1},\cdots,a_{n-1}]_{\omega}:=\left\{x\in \mathcal{E}_{\omega}:x_{i}=a_{i},i=0,1,\ldots,n-1\right\}.
\end{align*}
Let $$\Omega_{a}:=\{\omega:[a]_{\omega}\neq \emptyset\}=\{\omega:a\in \mathcal{W}^{n}(\omega)\}.$$ The set $\mathcal{W}^{n}$ is defined as the set of words $a$ of length $n$ with $\mathbb{P}(\Omega_{a})>0$, and $\mathcal{W}^{\infty}=\bigcup_{n\geq 1}\mathcal{W}^{n}$. For any $a\in\mathcal{W}^{\infty}(\omega), b\in\mathcal{W}^{\infty}(\vartheta^{n}\omega)$ with $n\geq |a|$, set
\begin{align*}
\mathcal{W}^{n}_{a,b}(\omega)=\left\{(x_{0},\ldots,x_{n-1})\in\mathcal{W}^{n}(\omega):(x_{0},\ldots,x_{|a|})=a,x_{n-1}b \ \text{is} \ \vartheta^{n-1}\omega\text{-admissible}\right\}.
\end{align*}
In the following, we introduce some dynamical properties of random dynamical systems. In this paper, we mainly consider the topologically mixing and relative big images and preimages property, which were discussed in \cite{Denker2008Thermodynamic,Kifer2008Thermodynamic,Stadlbauer2010On,Stadlbauer2017Coupling} respectively.
\begin{itemize}
  \item \emph{Topologically mixing}: for all $a,b\in \mathcal{W}^{1}$, there exists $N_{ab}\in \mathbb{N}$ such that if for $\mathbb{P}$-a.e. $\omega\in \Omega$, $n\geq N_{ab}$, $a\in S(\omega)$ and $b\in S(\vartheta^{n}\omega)$, then $\mathcal{W}^{n}_{a,b}(\omega)\neq\emptyset$.
  \item \emph{Relative big images and preimages property}: if there exists a family $\{\mathcal{I}_{bip}(\omega),\omega\in \Omega\}$ where $\mathcal{I}_{bip}(\omega)\subset \mathcal{W}^{1}(\omega)$ such that for $\mathbb{P}$-a.e. $\omega\in\Omega$, and for any $a\in\mathcal{W}^{1}(\omega)$, there exists $b\in \mathcal{I}_{bip}(\vartheta\omega)$ with $ab\in\mathcal{W}^{2}(\omega)$ or $c\in \mathcal{I}_{bip}(\vartheta^{-1}\omega)$ with $ca\in\mathcal{W}^{2}(\vartheta^{-1}\omega)$, we say that $f$ has relative big images and preimages (BIP) property.
\end{itemize}
Moreover, if $f$ is topologically mixing, then for the above words $a$ and $b$, the $ab$-element $\alpha_{ab}^{(n)}(\omega)$ of the matrix $A_{\omega}^{n}=A_{\vartheta^{n-1}\omega}A_{\vartheta^{n-2}\omega}\cdots A_{\vartheta\omega}A_{\omega}$ is positive. In this paper, we always assume that $f$ is topologically mixing, and $f$ has relative BIP property if $f$ is topologically mixing and $f$ has relative big images and preimages property.

\subsection{Relative Gurevi\v{c} pressure}

For a function $\varphi:\mathcal{E}\rightarrow\mathbb{R}$, set
\begin{align*}
V_{n}^{\omega}(\varphi):=\sup\left\{|\varphi(\omega,x)-\varphi(\omega,y)|:x_{i}=y_{i},i=0,1,\ldots,n-1\right\}.
\end{align*}
We say that $\varphi$ is locally fiber H\"{o}lder continuous if there exists a random variable $\kappa:\Omega\rightarrow\mathbb{N}$ where $\kappa(\omega)\geq 1$ with $\int_{\Omega} \kappa(\omega)d\mathbb{P}(\omega)<\infty$ such that for $\mathbb{P}$-a.e. $\omega\in \Omega$ and any $n\geq 1$, $V_{n}^{\omega}(\varphi)\leq \kappa(\omega)\cdot2^{-n}$. Let $\varphi:\mathcal{E}\rightarrow\mathbb{R}$ be a locally fiber H\"{o}lder continuous function. For $a,b\in\mathcal{W}^{1}$, $\omega\in \Omega_{a}$, $n\in \mathbb{N}$ and $\vartheta^{n}\omega\in \Omega_{b}$, the \emph{$n$-th random partition function} is defined by
\begin{align*}
Z_{n}^{\omega}(\varphi,a,b):=\sum_{\alpha\in\mathcal{W}^{n}_{a,b}(\omega)}\exp\left(\sup_{y\in [\alpha]_{\omega}}\sum_{i=0}^{n-1}\varphi\circ\Theta^{i}(\omega,y)\right),
\end{align*}
where we use the convention that $Z_{n}^{\omega}(\varphi,a,b)=0$ if $\mathcal{W}^{n}_{ab}(\omega)=\emptyset$. Define
\begin{align*}
n_{a}^{(0)}(\omega)=0
\end{align*}
and inductively
\begin{align*}
n_{a}^{(k+1)}(\omega)=\min\left\{n>n_{a}^{(k)}(\omega):\vartheta^{n}\omega\in \Omega_{a}\right\}.
\end{align*}
Notice that the function $n_{a}^{(k)}(\omega)$ is measurable in $\Omega$ since for any $\beta>0$, we have
\begin{align*}
\left\{\omega\in \Omega:n_{a}^{(1)}(\omega)<\beta\right\}=\bigcup_{n\in\mathbb{N}, \ n<\beta}\{\omega\in \Omega:\vartheta^{n}\omega\in \Omega_{a}\}\in\mathcal{F},
\end{align*}
and inductively
\begin{align*}
\left\{\omega\in \Omega:n_{a}^{(k+1)}(\omega)<\beta\right\}&=\left\{\omega\in \Omega:\exists \ n>n_{a}^{(k)}(\omega),\vartheta^{n}\omega\in \Omega_{a} \ \text{such that} \ n<\beta\right\}\\&=\bigcup_{n\in\mathbb{N}, \ n<\beta}\left\{\omega:n>n_{a}^{(k)}(\omega)\right\}\cap\{\omega\in \Omega:\vartheta^{n}\omega\in \Omega_{a}\}\\&\in\mathcal{F}.
\end{align*}
By the Kac lemma \cite{Cornfeld1982Ergodic}, we have
\begin{align*}
\int_{\Omega_{a}}n_{a}^{(1)}(\omega)d\mathbb{P}(\omega)=1,
\end{align*}
and then $n_{a}^{(1)}(\omega)<\infty$ for $\mathbb{P}_{a}$-a.e. $\omega\in \Omega_{a}$ where $\mathbb{P}_{a}$ is the normalized restriction of $\mathbb{P}$ on $\Omega_{a}$.
For each measurable in $(\omega,x)$ function $\phi$ on $\mathcal{E}$, the \emph{random Perron-Frobenius operator} with respect to $f$ is denoted by
\begin{align*}
\mathcal{L}_{\varphi}^{\omega}\phi(\vartheta\omega,x):&=\sum_{f_{\omega}(y)=x, \ y\in \mathcal{E}_{\omega}}\exp(\varphi(\omega,y))\cdot \phi(\omega,y).
\end{align*}
For any $n\in\mathbb{N}$, we define
\begin{align*}
\mathcal{L}_{\varphi}^{\omega,n}\phi(\vartheta^{n}\omega,x)&:=\mathcal{L}_{\varphi}^{\vartheta^{n-1}\omega}\circ \mathcal{L}_{\varphi}^{\vartheta^{n-2}\omega}\circ\cdots\circ \mathcal{L}_{\varphi}^{\vartheta\omega}\circ\mathcal{L}_{\varphi}^{\omega}\phi(\vartheta^{n}\omega,x)\\&=\sum_{f_{\omega}^{n}(y)=x, \ y\in\mathcal{E}_{\omega}}\exp\left(\sum_{i=0}^{n-1}\varphi\circ\Theta^{i}(\omega,y)\right)\cdot \phi(\omega,y).
\end{align*}

Define the measurable function with respect to $\omega\in \Omega$ as  $$B_{1}^{\omega}(\varphi):=\exp\left(\sum_{k\geq 1}V_{k+1}^{\vartheta^{-k}\omega}(\varphi)\right).$$
Since $B_{1}^{\omega}(\varphi)$ is measurable in $\omega\in\Omega$ and
\begin{align*}
\sum_{k\geq 1}V_{k+1}^{\vartheta^{-k}\omega}(\varphi)\leq \sum_{k\geq 1}\kappa(\vartheta^{-k}\omega)\cdot 2^{-(k+1)},
\end{align*}
by using the condition of $\int_{\Omega}\kappa(\omega)d\mathbb{P}(\omega)<\infty$ we have
$$\int_{\Omega}\log B_{1}^{\omega}(\varphi)d\mathbb{P}(\omega)<\infty.$$
In the following, we introduce the definition of relative Gurevi\v{c} pressure of a topologically mixing random countable Markov shift in \cite[Theorem 3.2]{Denker2008Thermodynamic}.
\begin{definition}\label{relative Gurevic pressure of random countable Markov shifts}
Let $f$ be a topologically mixing random countable Markov shift. If $\varphi$ is a locally fiber H\"{o}lder continuous function such that
\begin{align}\label{the condition of locally fiber Holder continuous function}
\int_{\Omega}\|\varphi(\omega,\cdot)\|_{\infty}d\mathbb{P}(\omega)<\infty \ \ \ \text{and} \ \ \ \int_{\Omega}\left|\log\|\mathcal{L}_{\varphi}^{\omega}{\bf 1}\|_{\infty}\right|d\mathbb{P}(\omega)<\infty.
\end{align}
Then the relative (fiber) Gurevi\v{c} pressure of the topologically mixing random countable Markov shift is defined by
\begin{align*}
\Pi_{Gur}^{(r)}(f,\varphi):=\int_{\Omega}\lim_{j\rightarrow\infty}\frac{1}{n_{a}^{(j)}(\omega)}\log Z_{n_{a}^{(j)}(\omega)}^{\omega}(\varphi,a,a)d\mathbb{P}(\omega),
\end{align*}
where the limit exists and it is constant for $\mathbb{P}_{a}$-a.e. $\omega\in \Omega_{a}$. Moreover, the limit is the same if we replace $a$ by any $b\in\mathcal{W}^{1}$. Moreover, for any $a,b\in\mathcal{W}^{1}$, one has
\begin{align*}
\Pi_{Gur}^{(r)}(f,\varphi)=\int_{\Omega}\lim_{j\rightarrow\infty}\frac{1}{n_{b}^{(j)}(\omega)}\log Z_{n_{b}^{(j)}(\omega)}^{\omega}(\varphi,a,b)d\mathbb{P}(\omega).
\end{align*}
\end{definition}
For the case of $\varphi=0$, the $h_{Gur}^{(r)}(f):=\Pi_{Gur}^{(r)}(f,0)$ is called relative Gurevi\v{c} entropy of the topologically mixing random countable Markov shift $f$.

\subsection{Random group extensions}

Let $f$ be a topologically mixing random countable Markov shift and $\varphi$ be a locally fiber H\"{o}lder continuous function such that
\begin{align*}
\int_{\Omega}\|\varphi(\omega,\cdot)\|_{\infty}d\mathbb{P}(\omega)<\infty \  \ \ \text{and} \ \ \ \int_{\Omega}\left|\log\|\mathcal{L}_{\varphi}^{\omega}1\|_{\infty}\right|d\mathbb{P}(\omega)<\infty.
\end{align*}
For a countable group $G$ and defining in $(\omega,x)$ map $\Psi:\mathcal{E}\rightarrow G, (\omega,x)\mapsto \psi_{\omega}(x)$ where $\psi_{\omega}:\mathcal{E}_{\omega}\rightarrow G$, $x\mapsto \psi_{\omega}(x)$ is induced by a map $\psi:\mathbb{N}^{\mathbb{N}}\rightarrow G$, that is $\psi_{\omega}$ is just the restriction of the action $\psi$ over $\mathcal{E}_{\omega}$ for $\mathbb{P}$-a.e. $\omega\in\Omega$ and $\psi$ only depends on one co-ordinate $x_{0}$ (in particular we set $\psi_{\omega}(x):=\psi_{\omega}(x_{0}):=\psi_{\omega}(x_{0},x_{1},\ldots,x_{n})$ for any $n\in\mathbb{N}$ where $x=(x_{0},x_{1},x_{2},\ldots)$), the \emph{random group extension or (random skew product extension)} $\mathcal{T}:=(\mathcal{T}_{\omega})_{\omega\in\Omega}$ is defined as follows:
\begin{align*}
\mathcal{T}_{\omega}:\mathcal{E}_{\omega}\times G\rightarrow\mathcal{E}_{\vartheta\omega}\times G, \ (x,g)\mapsto (f_{\omega}x,g\cdot\psi_{\omega}(x)), \ \omega\in \Omega,
\end{align*}
and the skew product transformation
\begin{align*}
\Theta_{\psi}:\mathcal{E}\times G\rightarrow\mathcal{E}\times G, \ (\omega,x,g)\mapsto (\vartheta\omega,f_{\omega}x,g\cdot\psi_{\omega}(x)).
\end{align*}

Denote by
\begin{align*}
\psi_{\omega}^{n}(x):=\psi_{\omega}(x)\cdot\psi_{\vartheta\omega}(f_{\omega}x)\cdot\cdots\cdot\psi_{\vartheta^{n-1}\omega}(f_{\omega}^{n-1}x)
\end{align*}
for any $n\geq 1$ and $x\in \mathcal{E}_{\omega}$. Note that $\mathcal{T}$ is a random countable Markov shift with measurably depending on $\omega\in \Omega$ matrices $$\tilde{A}_{\omega}=\left(\tilde{\alpha}_{(i,g),(j,h)}(\omega),(i,g)\in S(\omega)\times G,(j,h)\in S(\vartheta\omega)\times G\right)$$ with entries $\tilde{\alpha}_{(i,g),(j,h)}(\omega)\in \{0,1\}$ where $\tilde{\alpha}_{(i,g),(j,h)}(\omega)=1$ if $\alpha_{i,j}(\omega)=1$ and $\psi_{\omega}(i,j)=g^{-1}h$ and $\tilde{\alpha}_{(i,g),(j,h)}(\omega)=0$ otherwise when defining a measurable set
\begin{align*}
\tilde{\mathcal{E}}_{\omega}=\left\{\tilde{x}=((x_{0},g_{0}),(x_{1},g_{1}),\ldots):\tilde{\alpha}_{(x_{i},g_{i}),(x_{i+1},g_{i+1})}(\vartheta^{i}\omega)=1 \ \text{for any} \ i=0,1,\ldots\right\}
\end{align*}
with the left shifts $\tilde{f}_{\omega}:\tilde{\mathcal{E}}_{\omega}\rightarrow\tilde{\mathcal{E}}_{\vartheta\omega}$ by $((x_{0},g_{0}),(x_{1},g_{1}),\ldots)\mapsto((x_{1},g_{1}),(x_{2},g_{2}),\ldots)$.
By the definition of topologically mixing property, it follows that $\mathcal{T}$ is a topologically mixing random countable Markov shift if and only if for all $a,b\in \mathcal{W}^{1}$ and $g\in G$, there exists $N>0$ (depending on $a,b,g$) such that if for $\mathbb{P}$-a.e. $\omega\in \Omega$, $n\geq N$, $a\in S(\omega)$ and $b\in S(\vartheta^{n}\omega)$, then $g\in G_{a,b}^{n}(\omega)$ where $G_{a,b}^{n}(\omega)$ is defined by
\begin{align*} G_{a,b}^{n}(\omega):=\left\{\psi_{\omega}^{n}(\alpha):n\in\mathbb{N},\alpha\in\mathcal{W}^{n}(\omega),[a]_{\omega}\supset[\alpha]_{\omega},f_{\omega}^{n}([\alpha]_{\omega})\supset[b]_{\vartheta^{n}\omega}\right\}.
\end{align*}
Suppose that $\varphi$ is a locally fiber H\"{o}lder continuous function such that
\begin{align*}
\int_{\Omega}\|\varphi(\omega,\cdot)\|_{\infty}d\mathbb{P}(\omega)<\infty \ \ \ \text{and} \ \ \ \int_{\Omega}\left|\log\|\mathcal{L}_{\varphi}^{\omega}1\|_{\infty}\right|d\mathbb{P}(\omega)<\infty.
\end{align*}
Define $\tilde{\varphi}:\mathcal{E}\times G\rightarrow \mathbb{R},\tilde{\varphi}(\omega,x,g)=\varphi(\omega,x)$. Let $\pi_{1,2}:\mathcal{E}\times G\rightarrow \mathcal{E}$ denote the canonical projection. Then we have $\tilde{\varphi}=\varphi\circ \pi_{1,2}$. For a measurable function $v:\mathcal{E}\times G\rightarrow \mathbb{R}$, given $\omega\in \Omega$, the random Perron-Frobenius operator $\mathcal{L}_{\tilde{\varphi}}^{\omega}$ with respect to $\mathcal{T}$ is given by
\begin{align*}
\mathcal{L}_{\tilde{\varphi}}^{\omega}v(\vartheta\omega,x,g):&=\mathcal{L}_{\varphi\circ\pi_{1,2}}^{\omega}v(\vartheta\omega,x,g)\\&=\sum_{\mathcal{T}_{\omega}(y,h)=(x,g), \ (y,h)\in\mathcal{E}_{\omega}\times G}\exp(\varphi\circ\pi_{1,2}(\omega,y,h))\cdot v(\omega,y,h).
\end{align*}
If $\mathcal{T}$ is a topologically mixing random countable Markov shift, then the \emph{relative (fiber) Gurevi\v{c} pressure of the random group extension} is defined by
\begin{align}\label{the limit of the random group-extended Markov system}
\Pi_{Gur}^{(r)}(\mathcal{T},\tilde{\varphi}):=\int_{\Omega}\lim_{j\rightarrow\infty}\frac{1}{n_{a}^{(j)}(\omega)}\log \tilde{Z}_{n_{a}^{(j)}(\omega)}^{\omega}(\tilde{\varphi},a,a)d\mathbb{P}(\omega),
\end{align}
where
\begin{align*}
\tilde{Z}_{n}^{\omega}(\tilde{\varphi},a,a):=\sum_{\alpha\in\mathcal{W}^{n}_{a,a}(\omega), \ \psi_{\omega}^{n}(\alpha)=id}\exp\left(\sup_{y\in [\alpha]_{\omega}}\sum_{i=0}^{n-1}\tilde{\varphi}\circ(\Theta_{\psi})^{i}(\omega,y,id)\right).
\end{align*}

For the case of $\varphi=0$, the $h_{Gur}^{(r)}(\mathcal{T}):=\Pi_{Gur}^{(r)}(\mathcal{T},0)$ is called relative Gurevi\v{c} entropy of the random group $G$ extension.

\begin{remark}\label{the remark of random group extension}
For a random group extension with a random countable Markov shift, by \cite[Theorem 3.2]{Denker2008Thermodynamic} we have the following results.
\begin{itemize}
  \item[(1)] For a topologically mixing random group extension, the limit in \eqref{the limit of the random group-extended Markov system} exists and it is constant for $\mathbb{P}_{a}$-a.e. $\omega\in \Omega_{a}$. Moreover, the limit is the same if we replace $a$ by any $b\in\mathcal{W}^{1}$.
  \item[(2)] By the definitions of random group extension and random countable Markov shift, we have
  $\Pi_{Gur}^{(r)}(\mathcal{T},\tilde{\varphi})\leq \Pi_{Gur}^{(r)}(f,\varphi)$.
\end{itemize}
\end{remark}

Let $f$ be a topologically mixing random countable Markov shift and $\varphi$ be a locally fiber H\"{o}lder continuous function such that
\begin{align*}
\int_{\Omega}\|\varphi(\omega,\cdot)\|_{\infty}d\mathbb{P}(\omega)<\infty \ \ \text{and} \ \ \int_{\Omega}\left|\log\|\mathcal{L}_{\varphi}^{\omega}1\|_{\infty}\right|d\mathbb{P}(\omega)<\infty.
\end{align*}
If $f$ has relative BIP property and $\Pi_{Gur}^{(r)}(f,\varphi)<\infty$, by \cite[Remark 4.2]{Stadlbauer2010On} there exists a unique $\Theta$-invariant fiber (or relative) Gibbs measure $\mu$ which is defined by $d\mu(\omega,x)=d\mu_{\omega}(x)d\mathbb{P}(\omega)$ and a random variable $\lambda:\Omega\rightarrow(0,+\infty)$ satisfying $\int\log\lambda_{\omega}d\mathbb{P}(\omega)=\Pi_{Gur}^{(r)}(f,\varphi)$ such that there exists some random variable $C_{\varphi}:\Omega\rightarrow (0,+\infty)$ satisfying
\begin{align*}
\int_{\Omega} |\log C_{\varphi}(\omega)|d\mathbb{P}(\omega)<\infty
\end{align*}
such that
\begin{align*}
C_{\varphi}(\omega)^{-1}\leq \Lambda_{n}(\omega)\frac{\mu_{\omega}([x_{0},x_{1},\ldots,x_{n-1}]_{\omega})}{\exp\left(S_{n}\varphi(\omega,x)\right)}\leq C_{\varphi}(\omega),
\end{align*}
for any $x\in\mathcal{E}_{\omega}$ and $\mathbb{P}$-a.e. $\omega\in \Omega$, where $\{\mu_{\omega}\}_{\omega\in \Omega}$ are disintegrations of $\mu$ and $\Lambda_{n}(\omega):=\prod_{i=0}^{n-1}\lambda_{\vartheta^{i}\omega}$. For $p\in\mathbb{N}\cup \{\infty\}$ and measurable function $\{v(\omega,\cdot):\omega\in \Omega\}$ where $v(\omega,\cdot)\in L^{p}(\mathcal{E}_{\omega},\mathcal{B}_{\omega},\mu_{\omega})$ and $\omega\mapsto v(\omega,\cdot)$ is measurable, we denote by $\|v(\omega,\cdot)\|_{p}$ the $L^{p}$-norm of $v(\omega,\cdot)$. Set
\begin{align*}
\|v(\omega,\cdot,\cdot)\|_{\mathcal{H}_{p}(\omega)}=\left(\sum_{g\in G}\|v(\omega,\cdot,g)\|_{p}^{2}\right)^{\frac{1}{2}}
\end{align*}
and
\begin{align*}
\mathcal{H}_{p}(\omega)=\left\{v(\omega,\cdot,\cdot):\mathcal{E}_{\omega}\times G\rightarrow \mathbb{R}:\|v(\omega,\cdot,\cdot)\|_{\mathcal{H}_{p}(\omega)}<\infty\right\}.
\end{align*}

In particular, when $p=\infty$, we have
\begin{align*}
\|v(\omega,\cdot,\cdot)\|_{\mathcal{H}_{\infty}(\omega)}=\left(\sum_{g\in G}\|v(\omega,\cdot,g)\|_{\infty}^{2}\right)^{\frac{1}{2}}.
\end{align*}
\begin{remark}
For $v(\omega,\cdot,\cdot)\in \mathcal{H}_{p}(\omega)$, denote $v(\omega,\cdot,\cdot)$ by $v(\omega)$ for convenience. Then for any $\omega\in \Omega$, we observe that $$\mathcal{L}_{\tilde{\varphi}}^{\omega}:\mathcal{H}_{p}(\omega)\rightarrow\mathcal{H}_{p}(\vartheta\omega), \ v\mapsto\mathcal{L}_{\tilde{\varphi}}^{\omega}v(\vartheta\omega).$$
\end{remark}

For $\mathbb{P}$-a.e. $\omega\in\Omega$, the spectral radius of $\mathcal{L}_{\tilde{\varphi}}^{\omega}$ is defined by $${\rm spr}_{\mathcal{H}}(\mathcal{L}_{\tilde{\varphi}}^{\omega})=\limsup\limits_{n\rightarrow\infty}\|\mathcal{L}_{\tilde{\varphi}}^{\omega,n}\|_{\mathcal{H}_{\infty}(\vartheta^{n}\omega)}^{\frac{1}{n}}$$ where $$\|\mathcal{L}_{\tilde{\varphi}}^{\omega,n}\|_{\mathcal{H}_{\infty}(\vartheta^{n}\omega)}=\sup_{R\in \mathcal{H}_{\infty}(\omega), \ \|R\|_{\mathcal{H}_{\infty}(\omega)}= 1}\|\mathcal{L}_{\tilde{\varphi}}^{\omega,n}R(\vartheta^{n}\omega)\|_{\mathcal{H}_{\infty}(\vartheta^{n}\omega)}.$$

\subsection{Group extensions of random shifts of finite type}

In this subsection, we specialise to the case of finite sets $S(\omega),\omega\in \Omega$ when $l$ is a $\mathbb{N}$-valued random variable where $l(\omega)>1$ and $\int_{\Omega}\log l(\omega)d\mathbb{P}(\omega)<\infty$. The definition of random shifts of finite type was given in \cite{Kifer1995Multidimensional}. The corresponding definition was given in \cite{Kifer2006Random} related to matrices. Let $\mathcal{T}$ be a topologically mixing random group extension of a topologically mixing random shift of finite type $f$ by a countable group $G$. In this paper, we always assume that the random shift of finite type $f$ has topologically mixing property. Suppose that $\varphi:\mathcal{E}\rightarrow\mathbb{R}$ is a locally fiber H\"{o}lder continuous function satisfying \eqref{the condition of locally fiber Holder continuous function}. Note that $\mathcal{T}$ is a random countable Markov shift with measurably depending on $\omega\in \Omega$ matrices $$\tilde{A}_{\omega}=\left(\tilde{\alpha}_{(i,g),(j,h)}(\omega),(i,g)\in S(\omega)\times G,(j,h)\in S(\vartheta\omega)\times G\right)$$ with entries $\tilde{\alpha}_{(i,g),(j,h)}(\omega)\in \{0,1\}$ where $\tilde{\alpha}_{(i,g),(j,h)}(\omega)=1$ if $\alpha_{i,j}(\omega)=1$ and $\psi_{\omega}(i,j)=g^{-1}h$ and $\tilde{\alpha}_{(i,g),(j,h)}(\omega)=0$ otherwise. The relative Gurevi\v{c} pressure $\Pi_{Gur}^{(r)}(f,\varphi)$ of $f$ is equal to the relative (fiber) topological pressure $\Pi_{top}^{(r)}(f,\varphi)$ which is defined in \cite{Kifer2001On}. It follows from \cite[Proposition 1.6]{Kifer2001On} for $\mathbb{P}$-a.e. $\omega\in\Omega$, the relative topological pressure $\Pi_{top}^{(r)}(f,\varphi)$ is constant and the limit in exists. In addition, the limit equals again $\Pi_{top}^{(r)}(f,\varphi)$ with probability one if we take it along the subsequence $n_{a}^{(j)}(\omega),j\in\mathbb{N}$ in place of $n,n\in\mathbb{N}$, i.e.,
\begin{align*}
\Pi_{top}^{(r)}(f,\varphi)=\int_{\Omega}\lim_{j\rightarrow\infty}\frac{1}{n_{a}^{(j)}(\omega)}\log Z_{n_{a}^{(j)}(\omega)}^{\omega}(\varphi,a,a)d\mathbb{P}(\omega).
\end{align*}
Note that the random shift of finite type $f$ satisfies the relative BIP property and the corresponding relative Gurevi\v{c} pressure $\Pi_{Gur}^{(r)}(f,\varphi)<\infty$ by \cite[Remark 4.3]{Stadlbauer2010On}. Combining with Theorem \ref{the first result of main theorems}, then we have the following result.

\begin{remark}\label{the main remark of main results}
Let $\mathcal{T}$ be a topologically mixing random group extension of a random shift of finite type $f$ by a countable group $G$. The following statements hold.
\begin{itemize}
  \item Suppose that $\varphi:\mathcal{E}\rightarrow\mathbb{R}$ is a locally fiber H\"{o}lder continuous function satisfying \eqref{the condition of locally fiber Holder continuous function}. If $G$ is non-amenable, then for $\mathbb{P}$-a.e. $\omega\in \Omega$, one has
$$\log {\rm spr}_{\mathcal{H}}(\mathcal{L}_{\tilde{\varphi}}^{\omega})<\Pi_{Gur}^{(r)}(f,\varphi).$$
  \item Suppose that $\varphi:\mathcal{E}\rightarrow\mathbb{R}$ is a locally fiber H\"{o}lder continuous function satisfying \eqref{the condition of locally fiber Holder continuous function}. Then $\Pi_{Gur}^{(r)}(\mathcal{T},\tilde{\varphi})= \Pi_{Gur}^{(r)}(f,\varphi)$ implies that $G$ is amenable.
\end{itemize}
\end{remark}

Let $G^{\rm ab}$ be the abelianization of $G$ where $G^{\rm ab}=G/[G,G]$ and $[G,G]$ is a subgroup of $G$ generated by $$\left\{ghg^{-1}h^{-1}:g,h\in G\right\}.$$ The random group extension or (random skew product extension) $\mathcal{T}_{{\rm ab}}:=(\mathcal{T}_{{\rm ab}, \omega})_{\omega\in\Omega}$ is defined as follows:
\begin{align*}
\mathcal{T}_{{\rm ab}, \omega}:\mathcal{E}_{\omega}\times G^{{\rm ab}}\rightarrow\mathcal{E}_{\vartheta\omega}\times G^{{\rm ab}}, \ (x,g)\mapsto (f_{\omega}x,g\cdot\psi_{{\rm ab},\omega}(x)), \ \omega\in \Omega
\end{align*}
where $\psi_{{\rm ab},\omega}=\pi\circ\psi_{\omega}$ and $\pi:G\rightarrow G^{\rm ab}$ is the natural projection,
and the skew product transformation
\begin{align*}
\Theta_{\psi_{{\rm ab}}}:\mathcal{E}\times G^{{\rm ab}}\rightarrow\mathcal{E}\times G^{{\rm ab}}, \ (\omega,x,g)\mapsto (\vartheta\omega,f_{\omega}x,g\cdot\psi_{{\rm ab},\omega}(x)).
\end{align*}
Define the function $\tilde{\varphi}_{\rm ab}:\mathcal{E}\times G^{\rm ab}\rightarrow \mathbb{R},(\omega,x,g)\mapsto\tilde{\varphi}(\omega,x,g)$. Then the relative Gurevi\v{c} pressure of the random group $G^{\rm ab}$ extension is defined as
\begin{align*}
\Pi_{Gur}^{(r)}(\mathcal{T}_{\rm ab},\tilde{\varphi}_{\rm ab})&=\int_{\Omega}\lim_{j\rightarrow\infty}\frac{1}{n_{a}^{(j)}(\omega)}\log \tilde{Z}_{{\rm ab},n}^{\omega}(\tilde{\varphi},a,a)d\mathbb{P}(\omega)
\end{align*}
where
\begin{align*}
\tilde{Z}_{{\rm ab},n}^{\omega}(\tilde{\varphi}_{\rm ab},a,a):=\sum_{\alpha\in\mathcal{W}^{n}_{a,a}(\omega), \ \psi_{{\rm ab},\omega}^{n}(\alpha)=0}\exp\left(\sup_{y\in [\alpha]_{\omega}}\sum_{i=0}^{n-1}\tilde{\varphi}\circ(\Theta_{\psi_{\rm ab}})^{i}(\omega,y,id)\right).
\end{align*}
For the case of $\varphi=0$, the $h_{Gur}^{(r)}(\mathcal{T}_{\rm ab}):=\Pi_{Gur}^{(r)}(\mathcal{T}_{\rm ab},0)$ is called relative Gurevi\v{c} entropy of the random group $G^{\rm ab}$ extension.

\section{Random group extensions and its Perron-Frobenius operator}\label{Random group extensions and its Perron-Frobenius operator}

In this section, we give the proof of the part I of Theorem \ref{the first result of main theorems}, that is, Theorem \ref{the part I of the first result}. For a general case, it can be proved that for $\mathbb{P}$-a.e. $\omega\in \Omega$, then
\begin{align*}
\log {\rm spr}_{\mathcal{H}}(\mathcal{L}_{\tilde{\varphi}}^{\omega})\leq\Pi_{Gur}^{(r)}(f,\varphi)
\end{align*}
under the condition in Theorem \ref{the part I of the first result} for a general countable group $G$.

\begin{theorem}\label{the part I of the first result}
Let $\mathcal{T}$ be a topologically mixing random group extension of a random countable Markov shift $f$ with relative BIP property by a countable group $G$. Suppose that $\varphi:\mathcal{E}\rightarrow\mathbb{R}$ is a locally fiber H\"{o}lder continuous function satisfying \eqref{the condition of locally fiber Holder continuous function}. If $G$ is non-amenable, then for $\mathbb{P}$-a.e. $\omega\in \Omega$, one has
$$\log {\rm spr}_{\mathcal{H}}(\mathcal{L}_{\tilde{\varphi}}^{\omega})<\Pi_{Gur}^{(r)}(f,\varphi).$$
\end{theorem}

\begin{proof}[Proof of Theorem \ref{the part I of the first result}]
Firstly, for $\mathbb{P}$-a.e. $\omega\in \Omega$, we claim that $$\log {\rm spr}_{\mathcal{H}}(\mathcal{L}_{\tilde{\varphi}}^{\omega})\leq\Pi_{Gur}^{(r)}(f,\varphi).$$
And then in the following, fix $\omega\in \Omega$, we divide the proof into two cases.

{\bf Case 1.} Consider $\mathcal{L}_{\varphi}^{\omega}({\bf 1})={\bf 1}$ and then $\Pi_{Gur}^{(r)}(f,\varphi)=0$. Then for $\mathbb{P}$-a.e. $\omega\in \Omega$, the inequality in Theorem \ref{the part I of the first result} holds. Indeed,
by using (3) in Lemma \ref{the description of spectral of random group extension}, it can be proved that $${\rm spr}_{\mathcal{H}}(\mathcal{L}_{\tilde{\varphi}}^{\omega})=\limsup_{n\rightarrow\infty}\|T_{n}^{\omega}\|_{\mathcal{H}_{\infty}(\vartheta^{n}\omega)}^{\frac{1}{n}}.$$
Again, by using (2) in Lemma \ref{the description of spectral of random group extension}, we can obtain that $$\|T_{n}^{\omega}\|_{\mathcal{H}_{\infty}(\vartheta^{n}\omega)}\leq \|\mathcal{L}_{\tilde{\varphi}}^{\omega,n}\|_{\mathcal{H}_{\infty}(\vartheta^{n}\omega)}.$$
Combining with Proposition \ref{the estimation 1 of random operators} and above inequality, we have
\begin{align*}
{\rm spr}_{\mathcal{H}}(\mathcal{L}_{\tilde{\varphi}}^{\omega})=\limsup_{n\rightarrow\infty}\|T_{n}^{\omega}\|_{\mathcal{H}_{\infty}(\vartheta^{n}\omega)}^{\frac{1}{n}}\leq 1.
\end{align*}
This implies that $$\log {\rm spr}_{\mathcal{H}}(\mathcal{L}_{\tilde{\varphi}}^{\omega})\leq 0=\Pi_{Gur}^{(r)}(f,\varphi).$$

{\bf Case 2.} If not for Case 1, there exists locally fiber H\"{o}lder continuous $h:\mathcal{E}\rightarrow \mathbb{R}^{+}$ which $h(\omega,\cdot),\omega\in \Omega$ are bounded away from zero and infinity such that
\begin{align*}
\mathcal{L}_{\varphi}^{\omega}(h(\omega,\cdot))=\lambda_{\omega}\cdot h(\vartheta\omega,\cdot)
\end{align*}
where $\int\log\lambda_{\omega}d\mathbb{P}(\omega)=\Pi_{Gur}^{(r)}(f,\varphi)$ for $\mathbb{P}$-a.e. $\omega\in \Omega$.
Define $\varphi_{0}(\omega,\cdot)=\varphi(\omega,\cdot)+\log h(\omega,\cdot)-\log h\circ\Theta(\omega,\cdot)-\log \lambda_{\omega}.$ Then we have
$\mathcal{L}_{\varphi_{0}}^{\omega}({\bf 1})={\bf 1}$
and then $\Pi_{Gur}^{(r)}(f,\varphi_{0})=0$. By using Case 1, we have
\begin{align*}
\log {\rm spr}_{\mathcal{H}}(\mathcal{L}_{\tilde{\varphi}_{0}}^{\omega})\leq\Pi_{Gur}^{(r)}(f,\varphi_{0}).
\end{align*}
For any $v\in \mathcal{H}_{\infty}(\omega)$, we have
\begin{align*}
\mathcal{L}_{\tilde{\varphi}_{0}}^{\omega}v=\exp\left(-\log \lambda_{\omega}\right)\cdot\frac{1}{h\circ\pi_{1,2}(\vartheta\omega)}\cdot\mathcal{L}_{\tilde{\varphi}}^{\omega}v\cdot h\circ\pi_{1,2}.
\end{align*}
Therefore, we obtain that $\mathcal{L}_{\tilde{\varphi}_{0}}^{\omega}$ and $\exp\left(-\log \lambda_{\omega}\right)\cdot\mathcal{L}_{\tilde{\varphi}}^{\omega}$ have the same spectrum since $h$ is bounded. Hence,
$$\log {\rm spr}_{\mathcal{H}}(\mathcal{L}_{\tilde{\varphi}_{0}}^{\omega})=\log {\rm spr}_{\mathcal{H}}(\mathcal{L}_{\tilde{\varphi}}^{\omega})-\Pi_{Gur}^{(r)}(f,\varphi).$$ By Case 1, we have $$\log {\rm spr}_{\mathcal{H}}(\mathcal{L}_{\tilde{\varphi}_{0}}^{\omega})\leq \Pi_{Gur}^{(r)}(f,\varphi_{0}),$$ then
\begin{align*}
\log {\rm spr}_{\mathcal{H}}(\mathcal{L}_{\tilde{\varphi}}^{\omega})-\Pi_{Gur}^{(r)}(f,\varphi)\leq 0=\Pi_{Gur}^{(r)}(f,\varphi_{0}).
\end{align*}
Therefore,
\begin{align*}
\log {\rm spr}_{\mathcal{H}}(\mathcal{L}_{\tilde{\varphi}}^{\omega})\leq\Pi_{Gur}^{(r)}(f,\varphi).
\end{align*}
Finally, if the result in Theorem \ref{the part I of the first result} does not hold, by using Case 2, we only consider the case of the equality. If there exists a subset $\Lambda$ of $\Omega$ with $\mathbb{P}(\Lambda)>0$ such that for any $\omega\in \Lambda$, one has $$\log {\rm spr}_{\mathcal{H}}(\mathcal{L}_{\tilde{\varphi}}^{\omega})=\Pi_{Gur}^{(r)}(f,\varphi),$$ by using the similar process of Step 1 in Theorem \ref{the second result of main theorems}, it can be obtained that $G$ is amenable, which is a contradiction under the condition in the Theorem \ref{the part I of the first result}. Then we have for $\mathbb{P}$-a.e. $\omega\in \Omega$,
$$\log {\rm spr}_{\mathcal{H}}(\mathcal{L}_{\tilde{\varphi}}^{\omega})<\Pi_{Gur}^{(r)}(f,\varphi)$$ from the Claim 1 and Claim 2.
This implies the proof.
\end{proof}

\section{Random group extensions by amenable groups}\label{Random group extensions by amenable groups}

In this section, we give the proof of the part II of Theorem \ref{the first result of main theorems}, that is Theorem \ref{the part 2 of the first theorem}. Let $\mathcal{T}$ be a topologically mixing random group extension of a random countable Markov shift $f$ with relative BIP property by a countable group $G$, we mainly consider the question for relative Gurevi\v{c} pressures between the random group extension and the random countable Markov shift (Question \ref{problem of random group extension and random countable Markov shift}).
Define the set
\begin{align*}
\mathcal{H}_{c}(\omega):=\left\{R(\omega,\cdot,\cdot)\in\mathcal{H}_{\infty}(\omega):R(\omega,\cdot,\cdot) \ \text{is constant on} \ \mathcal{E}_{\omega}\times \{g\} \ \text{for any} \ g\in G\right\}.
\end{align*}
For any $\alpha\in \mathcal{W}^{n}(\omega)$, we consider the following map
\begin{align*}
\tau_{\alpha}:\bigcup_{\omega\in \Omega}\{\vartheta^{n}\omega\}\times f_{\omega}^{n}([\alpha]_{\omega})\rightarrow\bigcup_{\omega\in \Omega}\{\omega\}\times [\alpha]_{\omega}, \ (\vartheta^{n}\omega,x)\mapsto (\omega,y),
\end{align*}
where $f_{\omega}^{n}y=x$. Then
\begin{align*}
\varphi\circ\tau_{\alpha}:\bigcup_{\omega\in \Omega}\{\vartheta^{n}\omega\}\times f_{\omega}^{n}([\alpha]_{\omega})\rightarrow \mathbb{R}, \ (\vartheta^{n}\omega,x)\mapsto 1_{f_{\omega}^{n}([\alpha]_{\omega})}(x)\cdot\varphi\circ\tau_{\alpha}(\vartheta^{n}\omega,x).
\end{align*}
Set
\begin{align*}
\tilde{\tau}_{\alpha}:\bigcup_{\omega\in \Omega}\{\vartheta^{n}\omega\}\times f_{\omega}^{n}([\alpha]_{\omega})\times G\rightarrow \bigcup_{\omega\in \Omega}\{\omega\}\times [\alpha]_{\omega}\times G,
\end{align*}
where
\begin{align*}
(\vartheta^{n}\omega,x,g)\mapsto (\omega,y,g\cdot(\psi_{\omega}^{n}(y))^{-1})
\end{align*}
and $f_{\omega}^{n}y=x$. Denote $\tau_{\alpha}(\vartheta^{n}\omega,x)$ by $\tau_{\alpha}(\vartheta^{n}\omega)$, and $\tilde{\tau}_{\alpha}(\vartheta^{n}\omega,x,g)$ by $\tilde{\tau}_{\alpha}(\vartheta^{n}\omega)$ for convenience.
In addition, we have
\begin{align*}
\mathcal{L}_{\tilde{\varphi}}^{\omega}v(\vartheta\omega,x,g)=\sum_{\alpha\in\mathcal{W}^{1}(\omega)}\exp(\varphi\circ\tau_{\alpha}(\vartheta\omega,x))\cdot v\circ\tilde{\tau}_{\alpha}(\vartheta\omega,x,g).
\end{align*}

Next, we give the following result.

\begin{theorem}\label{the part 2 of the first theorem}
Let $\mathcal{T}$ be a topologically mixing random group extension of a random countable Markov shift $f$ with relative BIP property by a countable group $G$. Suppose that $\varphi:\mathcal{E}\rightarrow\mathbb{R}$ is a locally fiber H\"{o}lder continuous function satisfying \eqref{the condition of locally fiber Holder continuous function}. Then we have $\Pi_{Gur}^{(r)}(f,\varphi)=\Pi_{Gur}^{(r)}(\mathcal{T},\tilde{\varphi})$ implies that group $G$ is amenable.
\end{theorem}

In order to prove the Theorem \ref{the part 2 of the first theorem}, we need the following results. The case for a general topological group extension of a topological countable Markov shift was proved in \cite{Stadlbauer2013An}. We mainly follows the method of \cite{Stadlbauer2013An}. In the following, we consider the case of $\mathcal{L}_{\varphi}^{\omega}({\bf 1})={\bf 1}$ and then $\Pi_{Gur}^{(r)}(f,\varphi)=0$. The following result gives a description of random Perron-Frobenius operators $\mathcal{L}_{\tilde{\varphi}}^{\omega}$, $\omega\in\Omega$ with respect to $\mathcal{T}$, and estimates the relations between the random Perron-Frobenius operator and relative Gurevi\v{c} pressure.

\begin{proposition}\label{the estimation 1 of random operators}
Let $\mathcal{T}$ be a topologically mixing random group extension of a random countable Markov shift $f$ with relative BIP property by a countable group $G$. Suppose that $\varphi:\mathcal{E}\rightarrow\mathbb{R}$ is a locally fiber H\"{o}lder continuous function satisfying \eqref{the condition of locally fiber Holder continuous function}. The function spaces $(\mathcal{H}_{1}(\omega),\|\cdot\|_{\mathcal{H}_{1}(\omega)})$ and $(\mathcal{H}_{\infty}(\omega),\|\cdot\|_{\mathcal{H}_{\infty}(\omega)}),\omega\in \Omega$ are Banach spaces, the operators $$\mathcal{L}_{\tilde{\varphi}}^{\omega,k}:=\mathcal{L}_{\tilde{\varphi}}^{\vartheta^{k-1}\omega}\circ\cdots\circ\mathcal{L}_{\tilde{\varphi}}^{\vartheta\omega}\circ\mathcal{L}_{\tilde{\varphi}}^{\omega}:\mathcal{H}_{\infty}(\omega)\rightarrow\mathcal{H}_{\infty}(\vartheta^{k}\omega)$$ are bounded and there exists $C(\omega)>1$ such that $$\|\mathcal{L}_{\tilde{\varphi}}^{\omega,k}\|_{\mathcal{H}_{\infty}(\vartheta^{k}\omega)}\leq C(\omega)$$ for all $k\in\mathbb{N}$. Furthermore,
\begin{align*}
\mathcal{A}_{k}(\omega):&=\sup\left\{\frac{\|\mathcal{L}_{\tilde{\varphi}}^{\omega,k}R(\vartheta^{k}\omega,\cdot,\cdot)\|_{\mathcal{H}_{1}(\vartheta^{k}\omega)}}{\|R(\omega,\cdot,\cdot)\|_{\mathcal{H}_{1}(\omega)}}:R(\omega,\cdot,\cdot)\geq 0,R\in \mathcal{H}_{c}(\omega)\right\}\\
&\leq 1
\end{align*}
and then for $\mathbb{P}$-a.e. $\omega\in \Omega$, one has
\begin{align*}
\limsup_{k\rightarrow\infty}(\mathcal{A}_{k}(\omega))^{\frac{1}{k}}\geq\exp\left(\Pi_{Gur}^{(r)}(\mathcal{T},\tilde{\varphi})\right).
\end{align*}
\end{proposition}

\begin{proof}
The proof that $(\mathcal{H}_{p}(\omega),\|\cdot\|_{\mathcal{H}_{p}(\omega)})$ are Banach spaces by the definitions. To prove the uniform bound for $\|\mathcal{L}_{\tilde{\varphi}}^{\omega,k}\|_{\mathcal{H}_{\infty}(\vartheta^{k}\omega)}$, assume that $k\in\mathbb{N}$ and $R(\omega,\cdot,\cdot)\in\mathcal{H}_{\infty}(\omega)$. By using Jensen's inequality, set $$S_{k}\varphi(\omega,y)=\sum_{i=0}^{k-1}\varphi\circ \Theta^{i}(\omega,y)$$ for any $(\omega,y)\in\mathcal{E}$, we have
\begin{align*}
&\left\|\mathcal{L}_{\tilde{\varphi}}^{\omega,k}R(\vartheta^{k}\omega,\cdot,\cdot)\right\|_{\mathcal{H}_{\infty}(\vartheta^{k}\omega)}^{2}\\ \leq&\sum_{g\in G}\sup_{x\in\mathcal{E}_{\vartheta^{k}\omega}}\left(\sum_{\alpha\in\mathcal{W}^{k}(\omega)}\exp(S_{k}\varphi\circ\tau_{\alpha}(\vartheta^{k}\omega,x))\cdot R\circ \tilde{\tau}_{\alpha}(\vartheta^{k}\omega,x,g)\right)^{2}\\
\leq&\sum_{g\in G}\sup_{x\in\mathcal{E}_{\vartheta^{k}\omega}}\sum_{\alpha\in\mathcal{W}^{k}(\omega)}\exp(S_{k}\varphi\circ\tau_{\alpha}(\vartheta^{k}\omega,x))\cdot \left(\left\|R(\omega,\cdot,g\cdot(\psi_{\omega}^{k}(y))^{-1})\right\|_{\infty}\right)^{2}\\
\leq &C_{\varphi}(\omega)\cdot\sum_{\alpha\in\mathcal{W}^{k}(\omega)}\mu_{\omega}([\alpha]_{\omega})\cdot\|R(\omega,\cdot,\cdot)\|_{\mathcal{H}_{\infty}(\omega)}^{2}
\\=&C_{\varphi}(\omega)\cdot \|R(\omega,\cdot,\cdot)\|_{\mathcal{H}_{\infty}(\omega)}^{2},
\end{align*}
where $C_{\varphi}(\omega)$ is given by the fiber Gibbs property of $\mu_{\omega}$ on $\mathcal{E}$. Hence, $C_{\varphi}(\omega)$ is an upper bound for $\|\mathcal{L}_{\tilde{\varphi}}^{\omega,k}\|_{\mathcal{H}_{\infty}(\vartheta^{k}\omega)}$ independent of $k$. Assume that $R(\omega,\cdot,\cdot)\in\mathcal{H}_{c}(\omega)$, and set $S_{k}\tilde{\varphi}(\omega,y,g)=\sum_{i=0}^{k-1}\tilde{\varphi}\circ(\Theta_{\psi})^{i}(\omega,y,g)$ for any $(\omega,y,g)\in\mathcal{E}\times G$. It follows that
\begin{align*}
&\left\|\mathcal{L}_{\tilde{\varphi}}^{\omega,k}R(\vartheta^{k}\omega,\cdot,\cdot)\right\|_{\mathcal{H}_{1}(\vartheta^{k}\omega)}\\
\leq & \sum_{\alpha\in\mathcal{W}^{k}(\omega)}\left\|\exp(S_{k}\tilde{\varphi}\circ\tilde{\tau}_{\alpha}(\vartheta^{k}\omega,\cdot,\cdot))\cdot R\circ \tilde{\tau}_{\alpha}(\vartheta^{k}\omega,\cdot,\cdot)\right\|_{\mathcal{H}_{1}(\vartheta^{k}\omega)}\\
=&\sum_{\alpha\in\mathcal{W}^{k}(\omega)}\left(\sum_{g\in G}\left(\int\exp(S_{k}\tilde{\varphi}\circ\tilde{\tau}_{\alpha}(\vartheta^{k}\omega,\cdot,g))\cdot R\circ \tilde{\tau}_{\alpha}(\vartheta^{k}\omega,\cdot,g)d\mu_{\vartheta^{k}\omega}\right)^{2}\right)^{\frac{1}{2}}
\\=&\sum_{\alpha\in\mathcal{W}^{k}(\omega)}\left(\sum_{g\in G}\left(\mu_{\omega}([\alpha]_{\omega})\right)^{2}\cdot \left(\int R(\omega,\cdot,g\cdot (\psi_{\omega}^{k}(\alpha))^{-1})d\mu_{\omega}\right)^{2}\right)^{\frac{1}{2}}
\\ \leq& \left(\sum_{g\in G} \left(\int R(\omega,\cdot,g\cdot (\psi_{\omega}^{k}(\alpha))^{-1})d\mu_{\omega}\right)^{2}\right)^{\frac{1}{2}}
\\=&\|R(\omega,\cdot,\cdot)\|_{\mathcal{H}_{1}(\omega)}.
\end{align*}
Then by the above inequality, we have $$\mathcal{A}_{k}(\omega)\leq 1$$ for $k\in\mathbb{N}$. This obtains the result for $\mathcal{A}_{k}(\omega)$.
Observe that the fiber Gibbs property of $\mu_{\omega}$ implies that
\begin{align*}
&\left\|\mathcal{L}_{\tilde{\varphi}}^{\omega,k}{\bf 1}_{\{\omega\}\times \mathcal{E}_{\omega}\times \{id\}}(\vartheta^{k}\omega,\cdot,\cdot)\right\|_{\mathcal{H}_{1}(\vartheta^{k}\omega)} \geq\left\|\mathcal{L}_{\tilde{\varphi}}^{\omega,k}{\bf 1}_{\{\omega\}\times \mathcal{E}_{\omega}\times \{id\}}(\vartheta^{k}\omega,\cdot,id)\right\|_{\mathcal{H}_{1}(\vartheta^{k}\omega)}\\
\geq&\int\sum_{\alpha\in\mathcal{W}^{k}(\omega)}\exp(S_{k}\tilde{\varphi}\circ\tilde{\tau}_{\alpha}(\vartheta^{k}\omega,\cdot,id))\cdot {\bf 1}_{\{\omega\}\times \mathcal{E}_{\omega}\times \{id\}}\circ\tilde{\tau}_{\alpha}(\vartheta^{k}\omega,\cdot,id))d\mu_{\vartheta^{k}\omega}\\
\geq&\int \sum_{\alpha\in\mathcal{W}^{k}(\omega), \ \psi_{\omega}^{k}(\alpha)=id}\exp(S_{k}\varphi\circ\tau_{\alpha}(\vartheta^{k}\omega,\cdot))d\mu_{\vartheta^{k}\omega}
\geq\sum_{\alpha\in\mathcal{W}^{k}(\omega), \ \psi_{\omega}^{k}(\alpha)=id}\mu_{\omega}([\alpha]_{\omega})\\
\gg&\sum_{\alpha\in\mathcal{W}^{k}_{a,a}(\omega), \ \psi_{\omega}^{k}(\alpha)=id}\exp\left(\sup_{y\in[\alpha]_{\omega}}\sum_{i=0}^{k-1}\varphi\circ\Theta^{i}(\omega,y)\right)
\end{align*}
for all $a\in\mathcal{W}^{1}(\omega)$. This implies the proof.
\end{proof}

When $\Pi_{Gur}^{(r)}(\mathcal{T},\tilde{\varphi})=0$, by Proposition \ref{the estimation 1 of random operators}, we have
\begin{align*}
\limsup_{k\rightarrow\infty}(\mathcal{A}_{k}(\omega))^{\frac{1}{k}}\geq 1.
\end{align*}
Then we obtain the following result.

\begin{lemma}\label{the estimation 2 of random operators with functions}
Let $\mathcal{T}$ be a topologically mixing random group extension of a random countable Markov shift $f$ with relative BIP property by a countable group $G$. Suppose that $\varphi:\mathcal{E}\rightarrow\mathbb{R}$ is a locally fiber H\"{o}lder continuous function satisfying \eqref{the condition of locally fiber Holder continuous function}. Assume that $\Pi_{Gur}^{(r)}(\mathcal{T},\tilde{\varphi})=0$. Then there exists $n\in\mathbb{N}$ such that for any $\epsilon>0$, there exists a measurable function $\{R(\omega):\omega\in \Omega\}$ where $R(\omega)\in \mathcal{H}_{c}(\omega)$ with $R(\omega)\geq 0$ such that for $\mathbb{P}$-a.e. $\omega\in \Omega$, one has
\begin{align*}
\left\|\mathcal{L}_{\tilde{\varphi}}^{\omega,k}R(\vartheta^{k}\omega)-R(\vartheta^{k}\omega)\right\|_{\mathcal{H}_{1}(\vartheta^{k}\omega)}\leq \epsilon \cdot\|R(\omega)\|_{\mathcal{H}_{1}(\omega)}.
\end{align*}
\end{lemma}

\begin{proof}
Since $\mathcal{T}$ is a topologically mixing random group extension of a random countable Markov shift $f$ with relative BIP property by a countable group $G$, there exists $k\in\mathbb{N}$ and a measurable family $\{\mathcal{J}(\omega),\omega\in\Omega\}$ where $\mathcal{J}(\omega)\in\mathcal{W}^{k}(\omega)$ such that for $\mathbb{P}$-a.e. $\omega\in \Omega$, each pair $(\beta,\beta')$ with $\beta\in \mathcal{I}_{bip}(\omega)$ and $\beta'\in\mathcal{I}_{bip}(\vartheta^{k}\omega)$, there exists $u_{\beta,\beta'}\in \mathcal{J}(\omega)$ such that $u_{\beta,\beta'}\in \mathcal{W}^{k}(\omega)$ and $\psi_{\omega}^{k}(u_{\beta,\beta'})=id$. For $\omega\in \Omega$, define
\begin{align*}
\delta_{1}(\omega):=\inf\left\{\frac{\left\|\mathcal{L}_{\tilde{\varphi}}^{\omega,k}R(\vartheta^{k}\omega)-R(\vartheta^{k}\omega)\right\|_{\mathcal{H}_{1}(\vartheta^{k}\omega)}}{\|R(\omega)\|_{\mathcal{H}_{1}(\omega)}}:{ R(\omega)\in \mathcal{H}_{c}(\omega), R(\omega)\geq 0 \atop 0\neq R(\omega)\ \text{is measurable in} \ \Omega}\right\}.
\end{align*}

{\bf Step 1.} Firstly, there exists a finite set $\mathcal{W}^{*}(\omega)\subset \mathcal{W}^{k}(\omega)$ with
\begin{align*}
\sum_{\alpha\in\mathcal{W}^{k}(\omega)\setminus \mathcal{W}^{*}(\omega)}\mu_{\omega}([\alpha]_{\omega})\leq \frac{1}{4}\delta_{1}(\omega).
\end{align*}
For $R(\omega)\in\mathcal{H}_{c}(\omega)$, we have
\begin{align*}
&\left\|\sum_{\alpha\in\mathcal{W}^{k}(\omega)\setminus \mathcal{W}^{*}(\omega)}\exp(S_{k}\tilde{\varphi}\circ\tilde{\tau}_{\alpha}(\vartheta^{k}\omega))\cdot R\circ\tilde{\tau}_{\alpha}(\vartheta^{k}\omega)\right\|_{\mathcal{H}_{1}(\vartheta^{k}\omega)}\\
\leq&\sum_{\alpha\in\mathcal{W}^{k}(\omega)\setminus \mathcal{W}^{*}(\omega)}\mu_{\omega}([\alpha]_{\omega})\cdot\|R(\omega)\|_{\mathcal{H}_{1}(\omega)} \leq  \frac{1}{4}\delta_{1}(\omega)\cdot\|R(\omega)\|_{\mathcal{H}_{1}(\omega)}.
\end{align*}
Without loss of generality, we assume that $\mathcal{L}_{\varphi}^{\omega}({\bf 1})={\bf 1}$ and then $\Pi_{Gur}^{(r)}(f,\varphi)=0$. By using $\mathcal{L}_{\varphi}^{\omega}({\bf 1})={\bf 1}$, we have
\begin{align*}
&\sum_{\alpha\in\mathcal{W}^{*}(\omega)}\mu_{\omega}([\alpha]_{\omega})\cdot\left\|(R\circ\tilde{\tau}_{\alpha}(\vartheta^{k}\omega)-R(\vartheta^{k}\omega))\cdot {\bf 1}_{f_{\omega}^{k}([\alpha]_{\omega})\times G}\right\|_{\mathcal{H}_{1}(\vartheta^{k}\omega)} \\ \geq&\frac{1}{2}\delta_{1}(\omega)\cdot\|R(\omega)\|_{\mathcal{H}_{1}(\omega)}.
\end{align*}
It follows that there exists $\alpha_{R}\in\mathcal{W}^{*}(\omega)$ such that
\begin{align*}
\left\|(R\circ\tilde{\tau}_{\alpha_{R}}(\vartheta^{k}\omega)-R(\vartheta^{k}\omega))\cdot {\bf 1}_{f_{\omega}^{k}([\alpha_{R}]_{\omega})\times G}\right\|_{\mathcal{H}_{1}(\vartheta^{k}\omega)}\geq\frac{1}{2}\delta_{1}(\omega)\cdot\|R(\omega)\|_{\mathcal{H}_{1}(\omega)}.
\end{align*}
Moreover, the relative BIP property implies that $\mu_{\vartheta^{k}\omega}(f_{\omega}^{k}([\alpha]_{\omega}))$ is uniformly bounded from below for all $\alpha\in \mathcal{W}^{k}(\omega)$.
Then for any $x\in \mathcal{E}_{\omega}$, we have
\begin{align*}
&\left\|(R\circ\tilde{\tau}_{\alpha_{R}}(\vartheta^{k}\omega)-R(\vartheta^{k}\omega))\cdot {\bf 1}_{f_{\omega}^{k}([\alpha_{R}]_{\omega})\times G}\right\|_{\mathcal{H}_{1}(\vartheta^{k}\omega)}\\
\gg&\left\|R(\vartheta^{k}\omega,f_{\omega}^{k}x,\cdot)-R(\vartheta^{k}\omega,f_{\omega}^{k}x,\cdot(\psi_{\omega}^{k}(\alpha_{R}))^{-1})\right\|_{l^{2}(G)}.
\end{align*}

Next, for any $\omega\in\Omega$, we construct the $\omega$-admissible words.
Note that for any $x\in\mathcal{E}_{\omega}$, $\gamma\in\mathcal{I}_{bip}(\vartheta^{k}\omega)$ satisfying $[\gamma]_{\vartheta^{k}\omega}\subset f_{\omega}^{k}([\alpha_{R}]_{\omega})$, there exist $\alpha_{\gamma}\in\mathcal{J}(\omega)$ and $\alpha_{x}\in\mathcal{J}(\vartheta^{k}\omega)$ such that $f_{\omega}^{2k}x\in f_{\vartheta^{k}\omega}^{k}[\alpha_{x}]_{\vartheta^{k}\omega}$, $\alpha_{\gamma}\alpha_{x}$ and $\alpha_{R}\alpha_{x}$ are $\omega$-admissible, and then belongs to $\mathcal{W}^{2k}(\omega)$. Set
\begin{align*}
u=\alpha_{\gamma}\alpha_{x} \ \ \text{and} \ \ z=\alpha_{R}\alpha_{x}.
\end{align*}
Since $R(\omega,\cdot,\cdot)\in\mathcal{H}_{c}(\omega)$ and
\begin{align*}
\psi_{\omega}^{k}(\alpha_{\gamma})=\psi_{\vartheta^{k}\omega}^{k}(\alpha_{x})=id,
\end{align*}
we have
\begin{align*}
\psi_{\omega}^{2k}(u)=id \ \ \text{and} \ \ \psi_{\omega}^{2k}(z)=\psi_{\omega}^{k}(\alpha_{R}).
\end{align*}
The rotundity implies that there exists a unform constant $\delta_{2}(\omega)>0$ with
\begin{align*}
\frac{1}{2}\left\|R\circ\tilde{\tau}_{u}(\vartheta^{2k}\omega,f_{\omega}^{2k}x,\cdot)+R\circ\tilde{\tau}_{z}(\vartheta^{2k}\omega,f_{\omega}^{2k}x,\cdot)\right\|_{l^{2}(G)}\leq (1-\delta_{2}(\omega))\|R(\omega)\|_{\mathcal{H}_{1}(\omega)}
\end{align*}
for any $x\in\mathcal{E}_{\omega}$ satisfying $f_{\omega}^{2k}x\in f_{\vartheta^{k}\omega}^{k}[\alpha_{x}]_{\vartheta^{k}\omega}$ and $\mathbb{P}$-a.e. $\omega\in\Omega$.
By substituting $u$ and $z$ with $\alpha u$ and $\alpha z$, respectively, for some $\vartheta^{-(m-2)k}\omega$-admissible word $\alpha$ in $\mathcal{J}(\vartheta^{-(m-2)k}\omega)\times \mathcal{J}(\vartheta^{-(m-3)k}\omega)\times \cdots\times \mathcal{J}(\vartheta^{-2k}\omega)\times \mathcal{J}(\vartheta^{-k}\omega)$ and $m\in\mathbb{N}$ to be specified later,  there exists a family $\{\mathcal{W}^{\dag}(\omega),\omega\in \Omega\}$ where $\mathcal{W}^{\dag}(\omega)$ is a finite set and $\mathcal{W}^{\dag}(\omega)\subset \mathcal{W}^{mk}(\omega)$ with the following properties. For all $\alpha_{1}\in\mathcal{W}^{n_{1}}(\omega)$ and $\alpha_{2}\in\mathcal{W}^{n_{2}}(\vartheta^{(m-2)k+n_{1}}\omega)$, there exist $u(\alpha_{1},R,\alpha_{2})$ and $z(\alpha_{1},R,\alpha_{2})$ in $\mathcal{W}^{\dag}(\vartheta^{n_{1}}\omega)$ satisfying the following conditions:
\begin{itemize}
  \item[(1)] the corresponding estimate holds for $u:=u(\alpha_{1},R,\alpha_{2})$ and $z:=z(\alpha_{1},R,\alpha_{2})$ with $f_{\omega}^{\vartheta^{(m-2)k+n_{1}}}x\in f_{\vartheta^{(m-2)k+n_{1}}\omega}^{n_{1}}[\alpha_{x}]_{\vartheta^{(m-2)k+n_{1}}\omega}$ by taking $\omega$ with $\vartheta^{n_{1}}\omega$;
  \item[(2)] the first $(m-2)k$ letters of $u(\alpha_{1},R,\alpha_{2})$ and $z(\alpha_{1},R,\alpha_{2})$ coincide;
  \item[(3)] $\alpha_{1}u(\alpha_{1},R,\alpha_{2})\alpha_{2}$ and $\alpha_{1}z(\alpha_{1},R,\alpha_{2})\alpha_{2}$ are $\omega$-admissible.
\end{itemize}

{\bf Step 2.}
Secondly, choose $m\in\mathbb{N}$ such that for $\mathbb{P}$-a.e. $\omega\in \Omega$,
\begin{align*}
(1-\delta_{2}(\omega))^{\frac{1}{2}}&\leq \frac{\exp(S_{n}\varphi\circ \tau_{\alpha_{1}}(\tau_{u(\alpha_{1},R,\alpha_{2})}(\vartheta^{(m-2)k+n_{1}}\omega,f_{\omega}^{\vartheta^{(m-2)k+n_{1}}}x)))}{\exp(S_{n}\varphi\circ \tau_{\alpha_{1}}(\tau_{z(\alpha_{1},R,\alpha_{2})}(\vartheta^{(m-2)k+n_{1}}\omega,f_{\omega}^{\vartheta^{(m-2)k+n_{1}}}x)))}\\&\leq (1-\delta_{2}(\omega))^{-\frac{1}{2}}
\end{align*}
for any $n\in\mathbb{N}$ and $\alpha_{1}\in\mathcal{W}^{n}(\omega)$. Since $|\mathcal{W}^{\dag}(\omega)|<\infty,\omega\in \Omega$, we have
\begin{align*}
2\gamma(\omega):=\inf\left\{\exp(S_{mk}\varphi\circ\tau_{u}(\vartheta^{mk}\omega,y)):y\in f_{\omega}^{mk}([u]_{\omega}),u\in\mathcal{W}^{\dag}(\omega)\right\}>0.
\end{align*}
By dividing each $u\in\mathcal{W}^{\dag}(\omega)$ into two words $u_{1}$ and $u_{2}$ and setting $$\exp(S_{mk}\varphi\circ\tau_{u_{2}}(\vartheta^{mk}\omega,y)):=\exp(S_{mk}\varphi\circ\tau_{u}(\vartheta^{mk}\omega,y))-\gamma(\omega)$$ and $\exp(S_{mk}\varphi\circ\tau_{u_{1}}(\vartheta^{mk}\omega,y)):=\gamma(\omega)$ for each $y\in f_{\omega}^{mk}([u]_{\omega})$, without loss of generality we assume that
\begin{align*}
\exp(S_{mk}\varphi\circ\tau_{u}(\vartheta^{mk}\omega,y))=\gamma(\omega) \ \text{for any}  \ y\in f_{\omega}^{mk}([u]_{\omega}) \ \text{and} \ u\in\mathcal{W}^{\dag}(\omega).
\end{align*}

Given the function $R(\omega)\in\mathcal{H}_{c}(\omega)$ and $(k+1)$ finite $\omega$-admissible words $\alpha_{i}\in\mathcal{W}^{n_{i}}(\omega)$ for any $i=0,1,\ldots,k$. For $j=1,\ldots,k$, define $R_{j}(\vartheta^{s_{j}}\omega,f_{\omega}^{s_{j}}x,g)$ as
\begin{align*}
\sum_{(i_{1},i_{2},\ldots,i_{j-1})\in\{1,2\}^{j-1}}R\circ \tilde{\tau}_{\alpha_{0}u_{1}^{(i_{1})}\alpha_{1}\cdots u_{j-1}^{(i_{j-1})}\alpha_{j-1}}(\vartheta^{s_{j}}\omega,f_{\omega}^{s_{j}}x,g\cdot(\psi_{\omega}^{s_{j}}(x))^{-1})
\end{align*}
where $s_{j}:=n_{0}+\cdots+n_{j-1}+mk(j-1)$,
\begin{align*}
u_{1}^{(1)}:=u(\alpha_{0},R_{1},\alpha_{1}), \ \ u_{1}^{(2)}:=z(\alpha_{0},R_{1},\alpha_{1})\in\mathcal{W}^{\dag}(\vartheta^{n_{0}}\omega),
\end{align*}
and for $j=2,3,\ldots,k$,
\begin{align*}
u_{j}^{(1)}:=u(\alpha_{j-1},R_{j},\alpha_{j}), \ \ u_{j}^{(2)}:=z(\alpha_{j-1},R_{j},\alpha_{j})\in\mathcal{W}^{\dag}(\vartheta^{n_{0}+\cdots+n_{j-1}+mk(j-1)}\omega).
\end{align*}
Write $\hat{R}(g):=R(\omega,x,g)$ for $x\in\mathcal{E}_{\omega}$ and $R_{\alpha}(\omega,y,g):=R(\tau_{\alpha}(\vartheta^{n}\omega,y),g\cdot (\psi_{\omega}(\alpha))^{-1})$ for $y\in\mathcal{E}_{\vartheta^{n}\omega}$ and a finite word $\alpha\in\mathcal{W}^{n}(\omega)$. Let $N:=n_{0}+n_{1}+\cdots+n_{n}+nmk$. Then we have
\begin{align*}
&\left\|\sum_{(i_{1},\ldots,i_{n})\in\{1,2\}^{n}}\exp(S_{N}\varphi\circ\tau_{\alpha_{0}u_{1}^{(i_{1})}\alpha_{1}\cdots u_{n}^{(i_{n})}\alpha_{n}})\cdot R\circ\tilde{\tau}_{\alpha_{0}u_{1}^{(i_{1})}\alpha_{1}\cdots u_{n}^{(i_{n})}\alpha_{n}}(\vartheta^{N}\omega)\right\|_{l^{2}(G)}\\
\leq &\max_{(i_{1},i_{2},\ldots,i_{n})\in\{1,2\}^{n}}\exp(S_{N}\varphi\circ\tau_{\alpha_{0}u_{1}^{(i_{1})}\alpha_{1}\cdots u_{n}^{(i_{n})}\alpha_{n}}(\vartheta^{N}\omega))\cdot 2(1-\delta_{2}(\omega))\\
&\cdot \left\|\sum_{i_{n-1}\in\{1,2\}}\hat{R}_{n}(\cdot (\psi_{\vartheta^{N-mk-n_{n}}\omega}^{mk}(u_{n-1}^{(i_{n-1})}))^{-1})\right\|_{l^{2}(G)}.
\end{align*}
Combining with $S_{mk}\varphi\circ\tau_{u}(\vartheta^{mk}\omega,f_{\omega}^{mk}x)=\gamma(\omega)$ for all $x\in f_{\omega}^{mk}([u]_{\omega}) $ and $u\in\mathcal{W}^{\dag}(\omega)$, it follows that
\begin{align*}
&\left\|\sum_{(i_{1},\ldots,i_{n})\in\{1,2\}^{n}}\exp(S_{N}\varphi\circ\tau_{\alpha_{0}u_{1}^{(i_{1})}\alpha_{1}\cdots u_{n}^{(i_{n})}\alpha_{n}})\cdot R\circ\tilde{\tau}_{\alpha_{0}u_{1}^{(i_{1})}\alpha_{1}\cdots u_{n}^{(i_{n})}}(\vartheta^{N}\omega)\right\|_{l^{2}(G)}\\
\leq&\max_{(i_{1},\ldots,i_{n})\in\{1,2\}^{n}}\exp(S_{N}\varphi\circ\tau_{\alpha_{0}u_{1}^{(i_{1})}\alpha_{1}\cdots u_{n}^{(i_{n})}\alpha_{n}}(\vartheta^{N}\omega))\cdot (1-\delta_{3}(\omega))^{n}\cdot\|R(\omega)\|_{\mathcal{H}_{1}(\omega)},
\end{align*}
where $\delta_{3}(\omega):=1-(1-\delta_{2}(\omega))^{\frac{1}{2}}$.

{\bf Step 3.} Finally, fix $R(\omega)\in\mathcal{H}_{c}(\omega),\omega\in \Omega$ and $x\in\mathcal{E}_{\omega}$. Let $\mathcal{D}$ be the set of all subsets of $\{1,2,\ldots,n\}$. For each $\Lambda:=\{k_{1},k_{2},\ldots,k_{d}\}\in\mathcal{D}$, the subset $\mathcal{V}_{\Lambda}(\omega)$ of $\mathcal{W}^{nmk}(\omega)$ is defined as follows:
a word $\alpha=\alpha_{1}\alpha_{2}\cdots \alpha_{n}\in\mathcal{W}^{nmk}(\omega)$ is an element of $\mathcal{V}_{\Lambda}(\omega)$ if and only if there exist $\alpha_{k_{j}}^{(i_{j})}$, for $i_{j}=1,2$ and $j=1,2,\ldots,d$, such that $\alpha_{k_{j}}=\alpha_{k_{j}}^{(1)}$ and
\begin{align*}
&\left\|\sum_{\Gamma}\exp(S_{nmk}\varphi\circ\tau_{v})(\vartheta^{\sum_{i\notin \eta}n_{i}+mk|\eta|}\omega))\cdot R\circ\tilde{\tau}_{v}(\vartheta^{\sum_{i\notin \eta}n_{i}+mk|\eta|}\omega)\right\|_{l^{2}(G)}\\ \leq & \sum_{\Gamma}\exp(S_{nmk}\varphi\circ\tau_{v}(\vartheta^{\sum_{i\notin \eta}n_{i}+mk|\eta|}\omega))(1-\delta_{3}(\omega))^{|\eta|}\|R(\omega)\|_{\mathcal{H}_{1}(\omega)},
\end{align*}
where $\Gamma$ is taken over all $v=(v_{1},v_{2},\ldots,v_{n})$ with $v_{i}=\alpha_{i}$ for $i\notin \Lambda$ and $v_{i}\in\{\alpha_{i}^{(1)},\alpha_{i}^{(2)}\}$ for $i\in\Lambda$. Note that the construction of $u_{i}^{(1)}$ and $u_{i}^{(2)}$ and the estimation imply that $\{\mathcal{V}_{\Lambda}(\omega):\Lambda\in\mathcal{D}\}$ is a covering of $\mathcal{W}^{nmk}(\omega)$. Then
\begin{align*}
\mathcal{V}_{\Lambda}^{*}(\omega):=\mathcal{V}_{\Lambda}(\omega)\setminus \bigcup_{\Lambda\subset\Lambda',\Lambda\neq\Lambda'}\mathcal{V}_{\Lambda'}(\omega)
\end{align*}
defines a partition of $\mathcal{W}^{nmk}(\omega)$. Moreover, for $\Lambda=\{k_{1},k_{2},\ldots,k_{d}\}$, $j\in\{1,2,\ldots,d\}$ such that $k_{j}-1\notin \Lambda$ and $\alpha_{1}\alpha_{2}\cdots \alpha_{n}\in \mathcal{V}_{\Lambda}^{*}(\omega)$, one has
\begin{align*}
&\sum_{v_{1}\cdots v_{k_{j}-1}\alpha_{k_{j}}\cdots\alpha_{n}\in \mathcal{V}_{\Lambda}^{*}(\omega)}\exp(S_{mk}\varphi\circ\tau_{v_{k_{j-1}}\alpha_{k_{j}}\cdots\alpha_{n}}(\vartheta^{mk(k_{j}-2)}\omega,f_{\omega}^{mk(k_{j}-2)}x))\leq 1-2\gamma(\omega).
\end{align*}
It follows that
\begin{align*}
\left\|\mathcal{L}_{\varphi\circ\pi_{1,2}}^{\omega,nmk}R(\vartheta^{nmk}\omega,f^{nmk}_{\omega}(x),\cdot)\right\|_{l^{2}(G)}\leq (1-2\gamma(\omega)\delta_{3}(\omega))^{n}\cdot\left\|R(\omega)\right\|_{\mathcal{H}_{1}(\omega)}.
\end{align*}
This is a contradiction by using Jensen's inequality and Proposition \ref{the estimation 1 of random operators}. This implies that $\delta_{1}(\omega)=0$.

\end{proof}

In the following, we prove Theorem \ref{the part 2 of the first theorem}.

\begin{proof}[Proof of Theorem \ref{the part 2 of the first theorem}]
It suffices to prove the case of $\mathcal{L}_{\varphi}^{\omega}({\bf 1})={\bf 1}$ and then $\Pi_{Gur}^{(r)}(f,\varphi)=0$. If not, there exists locally fiber H\"{o}lder continuous $h:\mathcal{E}\rightarrow \mathbb{R}^{+}$ which $h(\omega,\cdot),\omega\in \Omega$ are bounded away from zero and infinity such that
\begin{align*}
\mathcal{L}_{\varphi}^{\omega}(h(\omega,\cdot))=\lambda_{\omega}\cdot h(\vartheta\omega,\cdot)
\end{align*}
where $\int\log\lambda_{\omega}d\mathbb{P}(\omega)=\Pi_{Gur}^{(r)}(f,\varphi)$.
Define $\varphi_{0}(\omega,\cdot):=\varphi(\omega,\cdot)+\log h(\omega,\cdot)-\log h\circ\Theta(\omega,\cdot)-\log\lambda_{\omega}$. Then we have
$\mathcal{L}_{\varphi_{0}}^{\omega}({\bf 1})={\bf 1}$
and then $\Pi_{Gur}^{(r)}(f,\varphi_{0})=0$. By the condition of $\Pi_{Gur}^{(r)}(f,\varphi)=\Pi_{Gur}^{(r)}(\mathcal{T},\tilde{\varphi})$, we have $\Pi_{Gur}^{(r)}(\mathcal{T},\tilde{\varphi}_{0})=\Pi_{Gur}^{(r)}(\mathcal{T},\tilde{\varphi})-\Pi_{Gur}^{(r)}(f,\varphi)=0$. Then we show the proof in the case of $\Pi_{Gur}^{(r)}(f,\varphi)=0$. Let $K$ be a finite subset of $G$. By the topologically mixing property of $\mathcal{T}$, there exists $m\in\mathbb{N}$ such that $m$ is a multiple of $n$ in Lemma \ref{the estimation 2 of random operators with functions} and $K\subset \{\psi_{\omega}^{m}(v):v\in \mathcal{W}^{m}(\omega)\}$ for $\mathbb{P}$-a.e. $\omega\in \Omega$. Then there exists a finite subset $\mathcal{W}_{K}(\omega)$ of $\mathcal{W}^{m}(\omega)$ with $K=\{\psi_{\omega}^{m}(v):v\in \mathcal{W}_{K}(\omega)\}$. By Lemma \ref{the estimation 2 of random operators with functions}, there exists a sequence of positive measurable function $\{R_{k}(\omega):\omega\in \Omega\}$ where $R_{k}(\omega)\in \mathcal{H}_{c}(\omega)$ with $R_{k}(\omega)\geq 0$ and $\|R_{k}(\omega)\|_{\mathcal{H}_{1}(\omega)}=1$ such that for $\mathbb{P}$-a.e. $\omega\in \Omega$, one has
\begin{align*}
\lim_{k\rightarrow\infty}\left\|\mathcal{L}_{\tilde{\varphi}}^{\omega,m}R_{k}(\vartheta^{m}\omega)-R_{k}(\vartheta^{m}\omega)\right\|_{\mathcal{H}_{1}(\vartheta^{m}\omega)}=0.
\end{align*}
Then $\lim_{k\rightarrow\infty}\left\|\mathcal{L}_{\tilde{\varphi}}^{\omega,m}R_{k}(\vartheta^{m}\omega)\right\|_{\mathcal{H}_{1}(\vartheta^{m}\omega)}=0$. We claim that for any $v\in \mathcal{W}_{K}(\omega)$, one has
$
\left\|\left(R_{k}\circ\tilde{\tau}_{v}(\vartheta^{m}\omega)-R_{k}(\vartheta^{m}\omega)\right)\cdot {\bf 1}_{f_{\omega}^{m}([v]_{\omega})\times G}\right\|_{\mathcal{H}_{1}(\vartheta^{m}\omega)}\rightarrow0.
$
If not, there exists $v\in \mathcal{W}_{K}(\omega)$ such that $$\liminf_{k\rightarrow\infty}\left\|\left(R_{k}\circ\tilde{\tau}_{v}(\vartheta^{m}\omega)-R_{k}(\vartheta^{m}\omega)\right)\cdot {\bf 1}_{f_{\omega}^{m}([v]_{\omega})\times G}\right\|_{\mathcal{H}_{1}(\vartheta^{m}\omega)}>0.$$
By the process in the first step of proof of Lemma \ref{the estimation 2 of random operators with functions}, one can imply that $$\left\|\mathcal{L}_{\tilde{\varphi}}^{\omega,m}R_{k}(\vartheta^{m}\omega)\right\|_{\mathcal{H}_{1}(\vartheta^{m}\omega)}$$ bounded away from 1, which is a contradiction. Choose a subsequence, for any $v\in \mathcal{W}_{K}(\omega)$, we have
\begin{align*}
\lim_{k\rightarrow\infty}\left\|\left(R_{k}\circ\tilde{\tau}_{v}(\vartheta^{m}\omega)-R_{k}(\vartheta^{m}\omega)\right)\cdot {\bf 1}_{f_{\omega}^{m}([v]_{\omega})\times G}\right\|_{\mathcal{H}_{1}(\vartheta^{m}\omega)}=0.
\end{align*}
Let $h\in G$. Then we have
\begin{align*}
\left\|R_{k}^{2}(\vartheta^{m}\omega,x,\cdot)-R_{k}^{2}(\vartheta^{m}\omega,x,\cdot h)\right\|_{1}
\leq 2\left\|R_{k}(\vartheta^{m}\omega,x,\cdot)-R_{k}(\vartheta^{m}\omega,x,\cdot h)\right\|_{2}.
\end{align*}
Fix $k\in\mathbb{N}$. There exists a random variable $p:\Omega\rightarrow\mathbb{N}\cup \{\infty\}$ with $p(\omega)\in\mathbb{N}\cup \{\infty\}$, $\lambda_{i}:\Omega\rightarrow (0,+\infty)$ and $A_{i}(\omega)\subset G$ with $A_{i}(\omega)\subset A_{i+1}(\omega)$ for any $1\leq i\leq p(\omega)$ such that $$R_{k}^{2}(\omega,y,\cdot)=\sum_{i=1}^{p(\omega)}\lambda_{i}(\omega){\bf 1}_{A_{i}(\omega)}(\cdot).$$
Note that $\sum_{i=1}^{p(\omega)}\lambda_{i}(\omega)|A_{i}(\omega)|=1$ and $(A_{i}(\omega)h\setminus A_{i}(\omega))\cap (A_{j}(\omega)\setminus A_{j}(\omega)h)=\emptyset$. It follows that
\begin{align*}
\left\|R_{k}^{2}(\vartheta^{m}\omega,x,\cdot)-R_{k}^{2}(\vartheta^{m}\omega,x,\cdot h)\right\|_{1}=\sum_{i=1}^{p(\vartheta^{m}\omega)}\lambda_{i}(\vartheta^{m}\omega)\left|A_{i}(\vartheta^{m}\omega)\triangle A_{i}(\vartheta^{m}\omega)h^{-1}\right|.
\end{align*}
For $\epsilon>0$, let $k>0$ satisfying
\begin{align*}
&\left\|\left(R_{k}\circ\tilde{\tau}_{v}(\vartheta^{m}\omega)-R_{k}(\vartheta^{m}\omega)\right)\cdot {\bf 1}_{f_{\omega}^{m}([v]_{\omega})\times G}\right\|_{\mathcal{H}_{1}(\vartheta^{m}\omega)}\\=&2\left\|R_{k}(\vartheta^{m}\omega,x,\cdot)-R_{k}(\vartheta^{m}\omega,x,\cdot h)\right\|_{2} \leq  |\mathcal{W}_{K}(\omega)|^{-1}\cdot\epsilon
\end{align*}
for any $v\in \mathcal{W}_{K}(\omega)$. Then
\begin{align*}
\frac{1}{2}\sum_{i=1}^{p(\vartheta^{m}\omega)}\lambda_{i}(\vartheta^{m}\omega)\sum_{h\in K}\left|A_{i}(\vartheta^{m}\omega)\triangle A_{i}(\vartheta^{m}\omega)h^{-1}\right|\leq \epsilon.
\end{align*}
It follows that there exists $1\leq i\leq p(\vartheta^{m}\omega)$ such that $$\sum_{h\in K}\left|A_{i}(\vartheta^{m}\omega)\triangle A_{i}(\vartheta^{m}\omega)h^{-1}\right|\leq 2\epsilon|A_{i}(\vartheta^{m}\omega)|.$$ Therefore, for any finite set $K$ and $\epsilon>0$, there exists a $(K,\epsilon)$-F{\o}lner set $A(\omega)$ which is finite and such that $$\sum_{h\in K}\left|A(\omega)h\triangle A(\omega)\right|\leq \epsilon|A(\omega)|.$$ This finishes the proof.
\end{proof}

\section{Amenability and random Perron-Frobenius operator}\label{Amenability and random Perron-Frobenius operator}

In this section, we mainly study the amenability and random Perron-Frobenius operator, and give the proof of Theorem \ref{the second result of main theorems}.

\begin{definition}
Define the linear operators $\mathcal{M}_{\omega}:\mathcal{H}_{\infty}(\omega)\rightarrow \mathcal{H}_{c}(\omega)$ and $T_{n}^{\omega}:\mathcal{H}_{c}(\omega)\rightarrow\mathcal{H}_{c}(\vartheta^{n}\omega)$ which for any $R(\omega)\in \mathcal{H}_{\infty}(\omega)$ and $n\in\mathbb{N}$ given by
\begin{align*}
\mathcal{M}_{\omega}(R(\omega)):=\sum_{g\in G}\left(\int R(\omega,\cdot,g)d\mu_{\omega}\right){\bf 1}_{\mathcal{E}_{\omega}\times \{g\}}
\end{align*}
and
\begin{align*}
T_{n}^{\omega}:=\mathcal{M}_{\vartheta^{n}\omega}\mathcal{L}_{\tilde{\varphi}}^{\omega,n}|_{\mathcal{H}_{c}(\omega)}.
\end{align*}
\end{definition}

\begin{remark}\label{the remark of linear operators and perron frobenius operators}
The above linear operators are positive and bounded with $\|\mathcal{M}_{\omega}\|_{\mathcal{H}_{\infty}(\omega)}= 1$ for $\mathbb{P}$-a.e $\omega\in \Omega$ and $$\|T_{n}^{\omega}\|_{\mathcal{H}_{\infty}(\vartheta^{n}\omega)}\leq \|\mathcal{M}_{\vartheta^{n}\omega}\|_{\mathcal{H}_{\infty}(\vartheta^{n}\omega)}\cdot\|\mathcal{L}_{\tilde{\varphi}}^{\omega,n}\|_{\mathcal{H}_{\infty}(\vartheta^{n}\omega)}=\|\mathcal{L}_{\tilde{\varphi}}^{\omega,n}\|_{\mathcal{H}_{\infty}(\vartheta^{n}\omega)}$$ for $\mathbb{P}$-a.e $\omega\in \Omega$ and any $n\in\mathbb{N}$.
\end{remark}

The following result gives the relations of operators $T_{n}^{\omega}$ and $\mathcal{L}_{\tilde{\varphi}}^{\omega}$. For a topological countable Markov shift, Jaerisch \cite{Jaerisch2015Groupextended} established these relations.

\begin{lemma}\label{the description of spectral of random group extension}
Let $\mathcal{T}$ be a topologically mixing random group extension of a random countable Markov shift $f$ with relative BIP property by a countable group $G$. Suppose that $\mathcal{L}_{\varphi}^{\omega}({\bf 1})={\bf 1}$, then we have the following results.
\begin{itemize}
  \item[(1)] $\|T_{n}^{\omega}\|_{\mathcal{H}_{\infty}(\vartheta^{n}\omega)}=\mathcal{A}_{n}(\omega)$ for any $n\in\mathbb{N}$.
  \item[(2)] $\|T_{n}^{\omega}\|_{\mathcal{H}_{\infty}(\vartheta^{n}\omega)}\leq \|\mathcal{L}_{\tilde{\varphi}}^{\omega,n}\|_{\mathcal{H}_{\infty}(\vartheta^{n}\omega)}\leq C_{\varphi}(\omega)\cdot\|T_{n}^{\omega}\|_{\mathcal{H}_{\infty}(\vartheta^{n}\omega)}$ for any $n\in\mathbb{N}$.
  \item[(3)] $\limsup\limits_{n\rightarrow\infty}\|T_{n}^{\omega}\|_{\mathcal{H}_{\infty}(\vartheta^{n}\omega)}^{\frac{1}{n}}=\limsup\limits_{n\rightarrow\infty}\left(\mathcal{A}_{n}(\omega)\right)^{\frac{1}{n}}={\rm spr}_{\mathcal{H}}(\mathcal{L}_{\tilde{\varphi}}^{\omega})$.
\end{itemize}
\end{lemma}

\begin{proof}
In the following, we give the proofs respectively.

(1). Let $n\in\mathbb{N}$. For any $R(\omega)\in\mathcal{H}_{c}^{+}(\omega):=\left\{R(\omega)\in\mathcal{H}_{c}(\omega):R(\omega)\geq 0\right\}$, we have
$$\|T_{n}^{\omega}R(\vartheta^{n}\omega)\|_{\mathcal{H}_{\infty}(\vartheta^{n}\omega)}=\|\mathcal{L}_{\tilde{\varphi}}^{\omega,n}R(\vartheta^{n}\omega)\|_{\mathcal{H}_{1}(\vartheta^{n}\omega)},$$
and by the Remark \ref{the remark of linear operators and perron frobenius operators}, this implies (1).

(2). By the Remark \ref{the remark of linear operators and perron frobenius operators}, we have $\|T_{n}^{\omega}\|_{\mathcal{H}_{\infty}(\vartheta^{n}\omega)}\leq \|\mathcal{L}_{\tilde{\varphi}}^{\omega,n}\|_{\mathcal{H}_{\infty}(\vartheta^{n}\omega)}$. Then we need to prove that
\begin{align*}
\|\mathcal{L}_{\tilde{\varphi}}^{\omega,n}\|_{\mathcal{H}_{\infty}(\vartheta^{n}\omega)}\leq C_{\varphi}(\omega)\cdot\|T_{n}^{\omega}\|_{\mathcal{H}_{\infty}(\vartheta^{n}\omega)}.
\end{align*}
Firstly, we claim that
\begin{align*}
\|\mathcal{L}_{\tilde{\varphi}}^{\omega,n}\|_{\mathcal{H}_{\infty}(\vartheta^{n}\omega)}=\sup_{R\in\mathcal{H}_{c}^{+}(\omega),\ \|R(\omega)\|_{\mathcal{H}_{\infty}(\omega)}=1}\left\{\|\mathcal{L}_{\tilde{\varphi}}^{\omega,n}R(\vartheta^{n}\omega)\|_{\mathcal{H}_{\infty}(\vartheta^{n}\omega)}\right\}.
\end{align*}
Indeed, for any $R(\omega)\in \mathcal{H}_{\infty}^{+}(\omega):=\left\{R(\omega)\in\mathcal{H}_{\infty}(\omega):R(\omega)\geq 0\right\}$, there exists $\hat{R}(\omega)\in \mathcal{H}_{c}^{+}(\omega),\omega\in\Omega$ where $\hat{R}(\omega,x,g):=\|R(\omega,\cdot,g)\|_{\mathcal{H}_{\infty}(\omega)}$ for $(x,g)\in \mathcal{E}_{\omega}\times G$ such that $\|\hat{R}(\omega)\|_{\mathcal{H}_{\infty}(\omega)}=\|R(\omega)\|_{\mathcal{H}_{\infty}(\omega)}$ and $$\|\mathcal{L}_{\tilde{\varphi}}^{\omega,n}R(\vartheta^{n}\omega)\|_{\mathcal{H}_{\infty}(\vartheta^{n}\omega)}\leq\|\mathcal{L}_{\tilde{\varphi}}^{\omega,n}\hat{R}(\vartheta^{n}\omega)\|_{\mathcal{H}_{\infty}(\vartheta^{n}\omega)}.$$ This finishes the claim.
For any $R(\omega)\in \mathcal{H}_{\infty}^{+}(\omega)$ and $y\in\mathcal{E}_{\omega}$, we have
\begin{align}\label{the equation 1 description of spectral of random group extension}
\|\mathcal{L}_{\tilde{\varphi}}^{\omega,n}R(\vartheta^{n}\omega,\cdot,g)\|_{\mathcal{H}_{\infty}(\vartheta^{n}\omega)} \leq C_{\varphi}(\omega)\cdot\sum_{\alpha\in\mathcal{W}^{n}(\omega)}\mu_{\omega}([\alpha]_{\omega})\cdot R(\omega,y,g\cdot(\psi_{\omega}^{n}(\alpha))^{-1}).
\end{align}
Since $\mu_{\omega}([\alpha])=\int \mathcal{L}_{\varphi}^{\omega,n}{\bf 1}_{\{\omega\}\times[\alpha]_{\omega}}(\vartheta^{n}\omega,\cdot)d\mu_{\vartheta^{n}\omega}$ for any $\alpha\in \mathcal{W}^{1}(\omega)$, then
\begin{align}\label{the equation 2 description of spectral of random group extension}
\sum_{\alpha\in \mathcal{W}^{n}(\omega)}\mu_{\omega}([\alpha]_{\omega})\cdot R(\omega,y,g\cdot(\psi_{\omega}^{n}(\alpha))^{-1})=\|\mathcal{L}_{\tilde{\varphi}}^{\omega,n}R(\vartheta^{n}\omega,\cdot,g)\|_{1}.
\end{align}
Then this implies (2) combining with \eqref{the equation 1 description of spectral of random group extension} and \eqref{the equation 2 description of spectral of random group extension}.

(3). By using (1) and (2), this implies (3) since the spectral radius of $\mathcal{L}_{\tilde{\varphi}}^{\omega}$ satisfies that $${\rm spr}_{\mathcal{H}}(\mathcal{L}_{\tilde{\varphi}}^{\omega})=\limsup\limits_{n\rightarrow\infty}\|\mathcal{L}_{\tilde{\varphi}}^{\omega,n}\|_{\mathcal{H}_{\infty}(\vartheta^{n}\omega)}^{\frac{1}{n}}.$$

\end{proof}

In the following, we prove the Theorem \ref{the second result of main theorems}.

\begin{proof}[Proof of Theorem \ref{the second result of main theorems}]
Given $\omega\in \Omega$. Assume that without loss of generality that $\mathcal{L}_{\varphi}^{\omega}({\bf 1})={\bf 1}$
and then $\Pi_{Gur}^{(r)}(f,\varphi)=0$. If not, there exists locally fiber H\"{o}lder continuous $h:\mathcal{E}\rightarrow \mathbb{R}^{+}$ which $h(\omega,\cdot),\omega\in \Omega$ are bounded away from zero and infinity such that
\begin{align*}
\mathcal{L}_{\varphi}^{\omega}(h(\omega,\cdot))=\lambda_{\omega}\cdot h(\vartheta\omega,\cdot)
\end{align*}
where $\int\log\lambda_{\omega}d\mathbb{P}(\omega)=\Pi_{Gur}^{(r)}(f,\varphi)$.
Define $\varphi_{0}(\omega,\cdot)=\varphi(\omega,\cdot)+\log h(\omega,\cdot)-\log h\circ\Theta(\omega,\cdot)-\log\lambda_{\omega}$. Then we have
$\mathcal{L}_{\varphi_{0}}^{\omega}({\bf 1})={\bf 1}$
and then $\Pi_{Gur}^{(r)}(f,\varphi_{0})=0$. Next, we divide the proof into two steps.

{\bf Step 1.} For any $v\in \mathcal{H}_{\infty}(\omega)$, we have
\begin{align*}
\mathcal{L}_{\tilde{\varphi}_{0}}^{\omega}v=\exp(-\log \lambda_{\omega})\cdot\frac{1}{h\circ\pi_{1,2}}\cdot\mathcal{L}_{\tilde{\varphi}_{0}}^{\omega}v\cdot h\circ\pi_{1,2}.
\end{align*}
Therefore, we obtain that $\mathcal{L}_{\tilde{\varphi}_{0}}^{\omega}$ and $\exp(-\log \lambda_{\omega})\cdot\mathcal{L}_{\tilde{\varphi}}^{\omega}$ have the same spectrum since $h$ is bounded. Hence,
$$\log {\rm spr}_{\mathcal{H}}(\mathcal{L}_{\tilde{\varphi}_{0}}^{\omega})=\log {\rm spr}_{\mathcal{H}}(\mathcal{L}_{\tilde{\varphi}}^{\omega})-\Pi_{Gur}^{(r)}(f,\varphi).$$ By the condition of the theorem \ref{the second result of main theorems}, we have
$\log {\rm spr}_{\mathcal{H}}(\mathcal{L}_{\tilde{\varphi}_{0}}^{\omega})=0$. Then we may assume that $\mathcal{L}_{\varphi}^{\omega}({\bf 1})={\bf 1}$.
By $\log {\rm spr}_{\mathcal{H}}(\mathcal{L}_{\tilde{\varphi}}^{\omega})=\Pi_{Gur}^{(r)}(f,\varphi)=0$ since $\mathcal{L}_{\varphi}^{\omega}({\bf 1})={\bf 1}$, then we have $${\rm spr}_{\mathcal{H}}(\mathcal{L}_{\tilde{\varphi}}^{\omega})= 1.$$ Therefore, by (3) in Lemma \ref{the description of spectral of random group extension}, we have
$$\limsup\limits_{n\rightarrow\infty}\left(\mathcal{A}_{n}(\omega)\right)^{\frac{1}{n}}=1.$$
It follows that $G$ is amenable by the processes of Proposition \ref{the estimation 1 of random operators}, Lemma \ref{the estimation 2 of random operators with functions} and Theorem \ref{the part 2 of the first theorem} under the condition of $\limsup\limits_{n\rightarrow\infty}\left(\mathcal{A}_{n}(\omega)\right)^{\frac{1}{n}}=1$.

{\bf Step 2.}  Suppose that $G$ is amenable. Observe that for any $n\in\mathbb{N}$, $R(\omega)\in \mathcal{H}_{c}(\omega)$ and $y\in\mathcal{E}_{\omega}$, one has
\begin{align*}
T_{n}^{\omega}R(\vartheta^{n}\omega)=\sum_{g\in G}\left(\sum_{\alpha\in \mathcal{W}^{n}(\omega)}\mu_{\omega}([\alpha]_{\omega})\cdot R(\omega,y,g\cdot(\psi_{\omega}^{n}(\alpha))^{-1})\right){\bf 1}_{\mathcal{E}_{\vartheta^{n}\omega}\times \{g\}},
\end{align*}
which is the Markov operator with respect to the distribution given by $$\mu_{\omega}\left(\left\{x\in\mathcal{E}_{\omega}:\psi_{\omega}^{n}(x_{0},x_{1},\ldots,x_{n-1})=g\right\}\right)$$ on $G$. Then we have $\|T_{n}^{\omega}\|_{\mathcal{H}_{\infty}(\vartheta^{n}\omega)}=1$ for $\mathbb{P}$-a.e. $\omega\in \Omega$ by \cite[Theorem 1]{Day1964Convolutions}. By using (3) in Lemma \ref{the description of spectral of random group extension}, we obtain that ${\rm spr}_{\mathcal{H}}(\mathcal{L}_{\tilde{\varphi}}^{\omega})=1$ for $\mathbb{P}$-a.e. $\omega\in\Omega$, which implies the proof.

\end{proof}

\section{Relative Gurevi\v{c} pressure for amenable group extensions}\label{Relative Gurevic pressure for amenable group extensions}

In this section, we consider the relative Gurevi\v{c} pressure for the case of amenable group $G$ and give the proof of Theorem \ref{the third result of main theorems}.
For each measurable in $(\omega,x)$ and continuous in $x\in \mathcal{E}_{\omega}$ function $\varphi$ on $\mathcal{E}$, let us set
\begin{align*}
\|\varphi\|_{1}=\int\|\varphi(\omega)\|_{\infty}d\mathbb{P}(\omega), \ \text{where} \ \|\varphi(\omega)\|_{\infty}=\sup_{x\in \mathcal{E}_{\omega}}|\varphi(\omega,x)|.
\end{align*}
Let $L^{1}_{\mathcal{E}}(\Omega,C(X))$ the space of such functions $\varphi$ with $\|\varphi\|_{1}<\infty$. For any subset $\Omega'\subset\Omega$, define $J_{\omega}(\Omega'):=\left\{n\in\mathbb{N}:\vartheta^{n}\omega\in\Omega'\right\}$. For $t>0$, $b_{n}\geq 0$, $a\in\mathcal{W}^{1}(\omega)$ and
given a measurable family $o=\{o(\omega)\in\mathcal{E}_{\omega}:o(\omega)\in[a]_{\omega}\}_{\omega\in \Omega}$, we let
\begin{align*}
\mathcal{P}_{a}(\omega,t):=\sum_{n\in J_{\omega}(\Omega_{a})}t^{-n}\cdot b_{n}\cdot \exp\left(\log \mathcal{Z}_{n}^{\omega}(\tilde{\varphi},a,a)\right)
\end{align*}
where for any $n\in\mathbb{N}$, $\mathcal{Z}_{n}^{\omega}(\tilde{\varphi},a,a)$ is defined by
\begin{align*}
\mathcal{Z}_{n}^{\omega}(\tilde{\varphi},a,a)=\sum_{\alpha\in\mathcal{W}^{n}_{a,a}(\omega)}\exp\left(S_{n}\varphi\circ\tau_{\alpha}(\vartheta^{n}\omega,o(\vartheta^{n}\omega))\cdot {\bf 1}_{\mathcal{E}\times \{id\}}\circ\Theta_{\psi}^{n}\left(\omega,y,id\right)\right).
\end{align*}
Observe that
\begin{align*}
\mathcal{P}_{a}(\omega,t)=\sum_{n\in \mathbb{N}}{\bf 1}_{\Omega_{a}}\circ \vartheta^{n}(\omega)\cdot t^{-n}\cdot b_{n}\cdot \exp\left(\log \mathcal{Z}_{n}^{\omega}(\tilde{\varphi},a,a)\right),
\end{align*}
and since the function ${\bf 1}_{\Omega_{a}}\circ \vartheta^{n}(\omega)\cdot t^{-n}\cdot b_{n}\cdot \exp\left(\log \mathcal{Z}_{n}^{\omega}(\tilde{\varphi},a,a)\right)$ is measurable in $\Omega$, we have that the function $\mathcal{P}_{a}(\omega,t)$ is measurable in $\Omega$. Then we define
\begin{align*}
\mathcal{P}_{a}(t):=\int_{\Omega} \mathcal{P}_{a}(\omega,t)d\mathbb{P}(\omega).
\end{align*}
If the sequence $\{b_{n}\}_{n\in\mathbb{N}}$ is omitted, then for a sequence $\left\{\log \mathcal{Z}_{n}^{\omega}(\tilde{\varphi},a,a)\right\}_{n\in\mathbb{N},n\in J_{\omega}(\Omega_{a})}$, by [\cite{Stadlbauer2010On}, Proposition 3.1], the number
\begin{align*}
\Pi_{Gur}^{(r)}(\mathcal{T},\tilde{\varphi})=\int_{\Omega}\lim_{n\rightarrow\infty,n\in J_{\omega}(\Omega_{a})}\frac{1}{n}\log \mathcal{Z}_{n}^{\omega}(\tilde{\varphi},a,a)d\mathbb{P}(\omega)
\end{align*}
is called the transition parameter of $\left\{\log \tilde{Z}_{n}^{\omega}(\tilde{\varphi},a,a)\right\}_{n\in\mathbb{N},n\in J_{\omega}(\Omega_{a})}$. Since $\mathbb{P}$ is ergodic, the above limit holds for $\mathbb{P}$-a.e. $\omega\in \Omega$ without taking integrals in the right-hand side. It is uniquely determined that by the fact that $\mathcal{P}_{a}(\omega,t)$ converge for $t>\exp\left(\Pi_{Gur}^{(r)}(\mathcal{T},\tilde{\varphi})\right)$ and diverges for $t<\exp\left(\Pi_{Gur}^{(r)}(\mathcal{T},\tilde{\varphi})\right)$ for $\mathbb{P}$-a.e. $\omega\in \Omega$. For $t=\exp\left(\Pi_{Gur}^{(r)}(\mathcal{T},\tilde{\varphi})\right)$, the sum may converge or diverge. By \cite[Lemma 3.1]{Denker1991On}, we have that there exists a sequence $\{b_{n}\}_{n\in\mathbb{N}}$ of positive reals such that $$\lim_{n\rightarrow\infty}\frac{b_{n}}{b_{n+1}}=1,$$
(also taking subsequence $\left\{n\right\}_{n\in\mathbb{N},n\in J_{\omega}(\Omega_{a})}$ if necessary), and in addition, for $\mathbb{P}$-a.e. $\omega\in \Omega$, $\mathcal{P}_{a}(\omega,t)$ has radius of convergence $\exp\left(\Pi_{Gur}^{(r)}(\mathcal{T},\tilde{\varphi})\right)$ and diverges at $\exp\left(\Pi_{Gur}^{(r)}(\mathcal{T},\tilde{\varphi})\right)$. Similarly, we can obtain that $\mathcal{P}_{a}(t)$ has radius of convergence $\exp\left(\Pi_{Gur}^{(r)}(\mathcal{T},\tilde{\varphi})\right)$ and diverges at $\exp\left(\Pi_{Gur}^{(r)}(\mathcal{T},\tilde{\varphi})\right)$. Write for $\rho=\exp\left(\Pi_{Gur}^{(r)}(\mathcal{T},\tilde{\varphi})\right)$.
For $t>\rho$, $c\in\mathcal{W}^{1}(\omega)$ and $g\in G$, given a measurable family $\eta=\{\eta(\omega)\in \mathcal{E}_{\omega}:\eta(\omega)\in [b]_{\omega}\}_{\omega\in \Omega}$, denote a fiber measure $\nu_{\omega,\eta,g}^{t}$ on $\{\omega\}\times\mathcal{E}_{\omega}\times G$ by
\begin{align*}
\int_{\{\omega\}\times\mathcal{E}_{\omega}\times G}R(\omega)d\nu_{\omega,\eta,g}^{t}:=\frac{1}{\mathcal{P}_{a}(t)}\sum_{n\in \mathbb{N}}t^{-n}\cdot b_{n}\cdot D_{n}^{\omega}(\varphi,R,\eta,g),
\end{align*}
for any $R\in L^{1}_{\mathcal{E}}(\Omega,C(X))$ where $D_{n}^{\omega}(\varphi,R,\eta,g)$ is defined as
\begin{align*}
\sum_{\alpha\in\mathcal{W}^{n}(\omega),\alpha\eta(\vartheta^{n}\omega) \ \omega\text{\rm -admissible}}\exp\left(S_{n}\varphi(\omega,\alpha\eta(\vartheta^{n}\omega))R\circ\tilde{\tau}_{\alpha}(\vartheta^{n}\omega,\eta(\vartheta^{n}\omega),g)\right).
\end{align*}
Since $D_{n}^{\omega}(\varphi,R,\eta,g)$ is measurable in $\Omega$, the function $\int_{\{\omega\}\times\mathcal{E}_{\omega}\times G}R(\omega)d\nu_{\omega,\eta,g}^{t}$ is measurable in $\Omega$.
Let $t\rightarrow \rho^{+}$, there exists a fiber measure $\nu_{\omega,\eta,g}$ on $\{\omega\}\times\mathcal{E}_{\omega}\times G$ since $\{\omega\}\times\mathcal{E}_{\omega}\times \{g\}$ is compact and $G$ is countable.

In order to prove Theorem \ref{the third result of main theorems}, we need the following lemmas.

\begin{lemma}\label{the lemma 1 of main result}
Let $\mathcal{T}$ be a topologically mixing random group extension of a random shift of finite type $f$ by a countable group $G$. The following three statements hold.
\begin{itemize}
  \item[(1)] There exists an non-random variable $C_{1}>0$ such that for any measurable families $\eta,\xi$ and elements $g,h\in G$ satisfying $f_{\omega}^{k}(\eta(\omega))=\xi(\vartheta^{k}\omega)$ and $g\cdot \psi_{\omega}^{k}(\eta(\omega))=h$, then for $\mathbb{P}$-a.e. $\omega\in \Omega$ we have
      $$\int_{\{\omega\}\times\mathcal{E}_{\omega}\times G}R(\omega)d\nu_{\omega,\eta,g}\leq C_{1}^{k}\cdot\int_{\{\omega\}\times\mathcal{E}_{\omega}\times G}R(\omega)d\nu_{\omega,\xi,h}.$$
  \item[(2)] There exists a random variable $C_{2}:\Omega\rightarrow (0,\infty)$ such that for  any measurable families $\eta,\xi$ and elements $g\in G$ where $\eta(\omega),\xi(\omega)$ belongs to the same cylinder of length $1$, then for $\mathbb{P}$-a.e. $\omega\in \Omega$ we have
      $$\int_{\{\omega\}\times\mathcal{E}_{\omega}\times G}R(\omega)d\nu_{\omega,\eta,g}\leq C_{2}(\omega)\cdot\int_{\{\omega\}\times\mathcal{E}_{\omega}\times G}R(\omega)d\nu_{\omega,\xi,g}.$$
  \item[(3)] There exists a random variable $C_{3}:\Omega\rightarrow (0,\infty)$ such that for any measurable families $\eta$ and $\xi$, then for $\mathbb{P}$-a.e. $\omega\in \Omega$ we have
      $$\int_{\{\omega\}\times\mathcal{E}_{\omega}\times G}R(\omega)d\nu_{\omega,\xi,g}\leq C_{3}(\omega)\cdot\int_{\{\omega\}\times\mathcal{E}_{\omega}\times G}R(\omega)d\nu_{\omega,\eta,g}.$$
\end{itemize}
\end{lemma}

\begin{proof}

(1). For any measurable families $\eta,\xi$ and elements $g,h\in G$ satisfying $f_{\omega}^{k}(\eta(\omega))=\xi(\vartheta^{k}\omega)$ and $g\cdot \psi_{\omega}^{k}(\eta(\omega))=h$, we have
\begin{align*}
\int_{\{\omega\}\times\mathcal{E}_{\omega}\times G}R(\omega)d\nu_{\omega,\xi,h}^{t}&=\frac{1}{\mathcal{P}_{a}(t)}\sum_{n\in \mathbb{N}}t^{-n}\cdot b_{n}\cdot D_{n}^{\omega}(\varphi,R,\xi,h)
\\
&\geq\frac{1}{\mathcal{P}_{a}(t)}\sum_{n\geq k}t^{-(n-k)} t^{-k} \frac{b_{n}}{b_{n-k}} b_{n-k}\cdot \left(B_{\varphi}\right)^{k}\cdot D_{n-k}^{\omega}(\varphi,R,\eta,g),
\end{align*}
where $B_{\varphi}=\inf_{x\in\mathcal{E}_{\omega},\omega\in\Omega}\exp(\varphi(\omega,x))$. Since
\begin{align*}
\int_{\{\omega\}\times\mathcal{E}_{\omega}\times G}R(\omega)d\nu_{\omega,\eta,g}^{t}=\frac{1}{\mathcal{P}_{a}(t)}\sum_{n=k}^{\infty}t^{-(n-k)}\cdot b_{n-k}\cdot D_{n-k}^{\omega}(\varphi,R,\eta,g),
\end{align*}
then
\begin{align*}
\int_{\{\omega\}\times\mathcal{E}_{\omega}\times G}R(\omega)d\nu_{\omega,\xi,h}^{t}\geq\left(\inf_{n\in \mathbb{N}}\frac{b_{n}}{b_{n-1}}\right)^{k}\cdot \left(B_{\varphi}\right)^{k}\cdot t^{-k}\cdot \int_{\{\omega\}\times\mathcal{E}_{\omega}\times G} R(\omega)d\nu_{\omega,\eta,g}^{t},
\end{align*}
Let $t\rightarrow \rho^{+}$, this gives (1) by taking $$C_{1}=\left(\inf_{n\in \mathbb{N}}\frac{b_{n}}{b_{n-1}}\cdot B_{\varphi}\cdot\frac{1}{\rho}\right)^{-1}.$$

(2). Fix any measurable families $\eta,\xi$ and elements $g\in G$ where $\eta(\omega),\xi(\omega)$ belongs to the same cylinder of length $1$. Let $R(\omega)={\bf 1}_{[u]_{\omega}}$ where $u\in \mathcal{W}^{k}(\omega)$ and $k\in\mathbb{N}$. For any $t>\rho$,
since
\begin{align*}
D_{n}^{\omega}(\varphi,R,\eta,g)=&\sum_{\alpha\in\mathcal{W}^{n}(\omega) \atop \alpha\eta(\vartheta^{n}\omega) \ \omega\text{\rm -admissible}}\exp\left(S_{n}\varphi(\omega,\alpha\eta(\vartheta^{n}\omega))R\circ\tilde{\tau}_{\alpha}(\vartheta^{n}\omega,\eta(\vartheta^{n}\omega),g\right)\\
\geq &\sum_{\alpha\in\mathcal{W}^{n}(\omega) \atop \alpha\xi(\vartheta^{n}\omega) \ \omega\text{\rm -admissible}}C_{2}(\omega)^{-1}\exp\left(S_{n}\varphi(\omega,\alpha\xi(\vartheta^{n}\omega))R\circ\tilde{\tau}_{\alpha}(\vartheta^{n}\omega,\xi(\vartheta^{n}\omega),g\right)\\
=&C_{2}(\omega)^{-1}\cdot D_{n}^{\omega}(\varphi,R,\xi,g)
\end{align*}
where
\begin{align*}
C_{2}:\Omega\rightarrow (0,+\infty), \ C_{2}(\omega)=\exp\left(2\kappa(\omega)\right),
\end{align*}
then we have
\begin{align*}
&\int_{\{\omega\}\times\mathcal{E}_{\omega}\times G} R(\omega)d\nu_{\omega,\eta,g}^{t}\\=&\frac{1}{\mathcal{P}_{a}(t)}\sum_{n\in \mathbb{N}}t^{-n}\cdot b_{n}\cdot D_{n}^{\omega}(\varphi,R,\eta,g)\\
\geq&\frac{1}{\mathcal{P}_{a}(t)}\sum_{n\in \mathbb{N}, \ n\geq k}t^{-n}\cdot b_{n}\cdot C_{3}(\omega)^{-1}\cdot D_{n}^{\omega}(\varphi,R,\xi,g)\\
=&C_{3}(\omega)^{-1}\left(\int_{\{\omega\}\times\mathcal{E}_{\omega}\times G}R(\omega)d\nu_{\omega,\xi,g}^{t}-\frac{1}{\mathcal{P}_{a}(t)}\sum_{n\in \mathbb{N}, \ n\leq k}t^{-n}\cdot b_{n}\cdot D_{n}^{\omega}(\varphi,R,\xi,g)\right).
\end{align*}
Since $\mathcal{P}_{a}(t)\rightarrow \infty$ when $t\rightarrow \rho^{+}$, this implies (2).

(3). By the mixing property of $\mathcal{T}$, for any measurable families $\eta$ and $\xi$ where for any $\omega\in \Omega$, $\eta(\omega)\in[a]_{\omega}$ and $\xi(\omega)\in [b]_{\omega}$ with $a,b\in\mathcal{W}^{1}$, there exists a measurable family $\zeta$ and $N\in\mathbb{N}$ such that $(\zeta(\omega),id)\in[b]_{\omega}\times \{g\}$ with $\mathcal{T}_{\omega}^{N}(\zeta(\omega),g)\in[a]_{\vartheta^{N}\omega}\times \{g\}$ for $\mathbb{P}$-a.e. $\omega\in \Omega$. Set $\mathcal{T}_{\omega}^{N}(\zeta(\omega),g):=(\zeta'(\vartheta^{N}\omega),g)$. Then $\zeta,\xi$ are in the same length $1$ cylinder and $\zeta',\eta$ are in the same length $1$ cylinder. By (1) we have
\begin{align}\label{inequality 1 of the lemma 1 of main result}
\int_{\{\omega\}\times\mathcal{E}_{\omega}\times G}R(\omega)d\nu_{\omega,\zeta,g}\leq C_{1}^{N}\cdot\int_{\{\omega\}\times\mathcal{E}_{\omega}\times G}R(\omega)d\nu_{\omega,\zeta',g}.
\end{align}
Again, by (2) we have
\begin{align}\label{inequality 2 of the lemma 1 of main result}
\int_{\{\omega\}\times\mathcal{E}_{\omega}\times G}R(\omega)d\nu_{\omega,\zeta',g}\leq C_{2}(\omega)\cdot\int_{\{\omega\}\times\mathcal{E}_{\omega}\times G}R(\omega)d\nu_{\omega,\eta,g}
\end{align}
and
\begin{align}\label{inequality 3 of the lemma 1 of main result}
\int_{\{\omega\}\times\mathcal{E}_{\omega}\times G}R(\omega)d\nu_{\omega,\xi,g}\leq C_{2}(\omega)\cdot\int_{\{\omega\}\times\mathcal{E}_{\omega}\times G}R(\omega)d\nu_{\omega,\zeta,g}.
\end{align}
Combining with \eqref{inequality 1 of the lemma 1 of main result}, \eqref{inequality 2 of the lemma 1 of main result} and \eqref{inequality 3 of the lemma 1 of main result}, it follows that
\begin{align*}
\frac{\int_{\{\omega\}\times\mathcal{E}_{\omega}\times G}R(\omega)d\nu_{\omega,\eta,g}}{\int_{\{\omega\}\times\mathcal{E}_{\omega}\times G}R(\omega)d\nu_{\omega,\xi,g}}&\geq (C_{2}(\omega))^{-2}\cdot\frac{\int_{\{\omega\}\times\mathcal{E}_{\omega}\times G}R(\omega)d\nu_{\omega,\zeta',g}}{\int_{\{\omega\}\times\mathcal{E}_{\omega}\times G}R(\omega)d\nu_{\omega,\zeta,g}}\\&\geq (C_{2}(\omega))^{-2}\cdot C_{1}^{-N}\\&\geq(C_{2}(\omega))^{-2}\cdot C_{1}^{-r},
\end{align*}
where $N$ can be bounded by some $r\in\mathbb{N}$ and
where
\begin{align*}
C_{3}:\Omega\rightarrow (0,+\infty), \ C_{3}(\omega)=(C_{2}(\omega))^{-2}\cdot C_{1}^{-r}.
\end{align*}
This implies (3).
\end{proof}

\begin{lemma}\label{the lemma 2 of main result}
Let $\mathcal{T}$ be a topologically mixing random group extension of a random shift of finite type $f$ by a countable group $G$. The following two results hold.
\begin{itemize}
  \item[(1)] For any $h\in G$, there is a random variable $C_{h}:\Omega\rightarrow \mathbb{R}$ such that for $\mathbb{P}$-a.e. $\omega\in \Omega$, we have
  $$\sup_{g\in G}\frac{\int_{\{\omega\}\times\mathcal{E}_{\omega}\times G}R(\omega)d\nu_{\omega,\eta,gh}}{\int_{\{\omega\}\times\mathcal{E}_{\omega}\times G}R(\omega)d\nu_{\omega,\eta,g}}=C_{h}(\omega).$$
  \item[(2)] There is a random variable $C_{L}:\Omega\rightarrow \mathbb{R}$ such that any measurable families $\eta,\xi$ and elements $g,h\in G$ satisfying $f_{\omega}^{k}(\eta(\omega))=\xi(\vartheta^{k}\omega)$ and $g\cdot \psi_{\omega}^{k}(\eta(\omega))=h$, then for $\mathbb{P}$-a.e. $\omega\in \Omega$ we have
  $$\int_{\{\omega\}\times\mathcal{E}_{\omega}\times G}R(\omega)d\nu_{\omega,\eta,g}\geq C_{L}(\omega)^{k}\cdot \int_{\{\omega\}\times\mathcal{E}_{\omega}\times G}R(\omega)d\nu_{\omega,\xi,h}.$$
\end{itemize}
\end{lemma}

\begin{proof}
(1).  Let $h\in G$, and fix a measurable family $\eta=\{\eta(\omega):\eta(\omega)\in[b]_{\omega}\}_{\omega\in \Omega}$ where $b\in\mathcal{W}^{1}$. Since $\mathcal{T}$ is topologically mixing, there exist a measurable family $\xi$ with $(\xi(\omega),h)\in [b]_{\omega}\times \{h\}$ and $N$ such that $f_{\omega}^{N}(\xi(\omega))\in [b]_{\vartheta^{N}\omega}$ and $h\cdot\psi_{\omega}^{N}(\xi(\omega))=id$. Therefore, for any $g\in G$, we have $(\xi(\omega),gh)\in[b]_{\omega}\times \{gh\}$, and $f_{\omega}^{N}(\xi(\omega))\in [b]_{\vartheta^{N}\omega}$ and $g\cdot h\cdot\psi_{\omega}^{N}(\xi(\omega))=g$. Let $\zeta$ be a measurable family where $\zeta(\vartheta^{N}\omega):=f_{\omega}^{N}(\xi(\omega))$ and $g:=gh\cdot\psi_{\omega}^{N}(\xi(\omega))$. By (1) in Lemma \ref{the lemma 1 of main result}, we have
\begin{align*}
\frac{\int_{\{\omega\}\times\mathcal{E}_{\omega}\times G}R(\omega)d\nu_{\omega,\xi,gh}}{\int_{\{\omega\}\times\mathcal{E}_{\omega}\times G}R(\omega)d\nu_{\omega,\zeta,g}}\leq C_{1}^{N}.
\end{align*}
By (2) in Lemma \ref{the lemma 1 of main result}, we have
\begin{align*}
&\frac{\int_{\{\omega\}\times\mathcal{E}_{\omega}\times G}R(\omega)d\nu_{\omega,\eta,gh}}{\int_{\{\omega\}\times\mathcal{E}_{\omega}\times G}R(\omega)d\nu_{\omega,\eta,g}}\\=&\frac{\int_{\{\omega\}\times\mathcal{E}_{\omega}\times G}R(\omega)d\nu_{\omega,\eta,gh}}{\int_{\{\omega\}\times\mathcal{E}_{\omega}\times G}R(\omega)d\nu_{\omega,\xi,gh}}\cdot\frac{\int_{\{\omega\}\times\mathcal{E}_{\omega}\times G}R(\omega)d\nu_{\omega,\zeta,g}}{\int_{\{\omega\}\times\mathcal{E}_{\omega}\times G}R(\omega)d\nu_{\omega,\eta,g}}\cdot \frac{\int_{\{\omega\}\times\mathcal{E}_{\omega}\times G}R(\omega)d\nu_{\omega,\xi,gh}}{\int_{\{\omega\}\times\mathcal{E}_{\omega}\times G}R(\omega)d\nu_{\omega,\zeta,g}}\\
\leq& C_{2}(\omega)^{2}\cdot C_{1}^{N}.
\end{align*}
Then this obtains (1).

(2). Given measurable families $\eta,\xi$ and elements $g,h\in G$ satisfying $f_{\omega}^{k}(\eta(\omega))=\xi(\vartheta^{k}\omega)$ and $g\cdot \psi_{\omega}^{k}(\eta(\omega))=h$. Suppose that $\eta(\omega)\in [b_{0},b_{1},\ldots,b_{k-1},b_{k}]_{\omega}$ where $b_{i}\in\mathcal{W}^{1}$. Therefore, we have $\psi_{\omega}^{k}(\eta(\omega))=\psi_{\omega}(b_{0})\cdots\psi_{\vartheta^{k-1}\omega}(b_{k-1})$ and then $f_{\omega}^{k}(\eta(\omega))=\xi(\vartheta^{k}\omega)$ and $g\cdot\psi_{\omega}(b_{0})\cdots\psi_{\vartheta^{k-1}\omega}(b_{k-1})=h$. Then $\xi(\vartheta^{k}\omega)\in [b_{k}]_{\vartheta^{k}\omega}$ and we have $g\cdot\psi_{\omega}(b_{0})\cdots\psi_{\vartheta^{k-1}\omega}(b_{k-1})=h$. Let $s_{i}:=\psi_{\vartheta^{i}\omega}(b_{i})$ for any $i=0,\ldots,k-1$. By using (3) in Lemma \ref{the lemma 1 of main result}, we have
\begin{align*}
\int_{\{\omega\}\times\mathcal{E}_{\omega}\times G}R(\omega)d\nu_{\omega,\xi,h}\leq C_{3}(\omega)\cdot\int_{\{\omega\}\times\mathcal{E}_{\omega}\times G}R(\omega)d\nu_{\omega,\eta,h}.
\end{align*}
And by (2) in Lemma \ref{the lemma 1 of main result}, one has
\begin{align*}
\frac{\int_{\{\omega\}\times\mathcal{E}_{\omega}\times G}R(\omega)d\nu_{\omega,\eta,g}}{\int_{\{\omega\}\times\mathcal{E}_{\omega}\times G}R(\omega)d\nu_{\omega,\xi,h}}&\geq C_{3}(\omega)^{-1}\frac{\int_{\{\omega\}\times\mathcal{E}_{\omega}\times G}R(\omega)d\nu_{\omega,\eta,g}}{\int_{\{\omega\}\times\mathcal{E}_{\omega}\times G}R(\omega)d\nu_{\omega,\eta,h}}\\&\geq C_{2}(\omega)^{-2}\cdot C_{3}(\omega)^{-1}\frac{\int_{\{\omega\}\times\mathcal{E}_{\omega}\times G}R(\omega)d\nu_{\omega,\zeta,g}}{\int_{\{\omega\}\times\mathcal{E}_{\omega}\times G}R(\omega)d\nu_{\omega,\zeta,h}},
\end{align*}
where $\zeta$ is a measurable family that $\zeta(\omega)$ belongs to the same cylinder as $\eta(\omega)$. By the statement of (1), we have
\begin{align*}
\frac{\int_{\{\omega\}\times\mathcal{E}_{\omega}\times G}R(\omega)d\nu_{\omega,\zeta,g}}{\int_{\{\omega\}\times\mathcal{E}_{\omega}\times G}R(\omega)d\nu_{\omega,\zeta,h}}&=\frac{\int_{\{\omega\}\times\mathcal{E}_{\omega}\times G}R(\omega)d\nu_{\omega,\zeta,g}}{\int_{\{\omega\}\times\mathcal{E}_{\omega}\times G}R(\omega)d\nu_{\omega,\zeta,gs_{0}\cdots s_{k-1}}}\\
&=\frac{\int_{\{\omega\}\times\mathcal{E}_{\omega}\times G}R(\omega)d\nu_{\omega,\zeta,g}}{\int_{\{\omega\}\times\mathcal{E}_{\omega}\times G}R(\omega)d\nu_{\omega,\zeta,gs_{0}}}\cdots\frac{\int_{\{\omega\}\times\mathcal{E}_{\omega}\times G}R(\omega)d\nu_{\omega,\zeta,gs_{0}\cdots s_{k-2}}}{\int_{\{\omega\}\times\mathcal{E}_{\omega}\times G}R(\omega)d\nu_{\omega,\zeta,gs_{0}\cdots s_{k-1}}}\\
&\geq\frac{1}{C_{s_{0}}(\omega)\cdot C_{s_{1}}(\omega) \cdots C_{s_{k-1}}(\omega)}.
\end{align*}
Let $\mathcal{S}=\{\psi(B):|B|=1\}$. Then
\begin{align*}
C_{\zeta}(\omega)=\min\left\{(C_{2}(\omega))^{-2}\cdot (C_{3}(\omega))^{-1}\cdot\frac{1}{C_{s}(\omega)}:s\in \mathcal{S}\right\}.
\end{align*}
Finally, this implies (2) by taking the minimum over the finitely many choices of $\zeta$.

\end{proof}

\begin{lemma}\label{the lemma 3 of main result}
For any random continuous function $R:\mathcal{E}\times G\rightarrow \mathbb{R}$ and any $g\in G$, then we have
\begin{align*}
&\rho\cdot\int_{\{\omega\}\times\mathcal{E}_{\omega}\times G}R(\omega)d\nu_{\omega,\eta,g}\\=&\sum_{\alpha\in\mathcal{W}^{1}(\omega)\atop\alpha\eta(\vartheta\omega) \ \omega\text{\rm -admissible}}\exp\left(\varphi(\omega,\alpha\eta(\vartheta\omega))\right)\int_{\{\omega\}\times\mathcal{E}_{\omega}\times G}R(\omega)d\nu_{\omega,\alpha\eta(\vartheta\omega),g\cdot (\psi_{\omega}(\alpha))^{-1}}.
\end{align*}
\end{lemma}

\begin{proof}
Fix $t>\rho$, it follows that
\begin{align*}
\int_{\{\omega\}\times\mathcal{E}_{\omega}\times G}R(\omega)d\nu_{\omega,\eta,g}^{t}=&\frac{1}{\mathcal{P}_{a}(t)}\sum_{n\in\mathbb{N}}t^{-n}\cdot b_{n}\cdot D_{n}^{\omega}(\varphi,R,\eta,g)\\
=&\frac{1}{\mathcal{P}_{a}(t)}\sum_{u\in \mathcal{W}^{1}(\omega) \atop u\eta(\vartheta\omega) \ \omega\text{\rm -admissible}}t^{-1}\exp(\varphi(\omega,u\eta(\vartheta\omega)))\sum_{n\in \mathbb{N}}b_{n-1}\\ & \cdot \frac{b_{n}}{b_{n-1}}\cdot t^{-(n-1)}\cdot D_{n-1}^{\omega}(\varphi,R,u\eta,g).
\end{align*}
Considering $t\rightarrow \rho^{+}$, we have
\begin{align*}
&\int_{\{\omega\}\times\mathcal{E}_{\omega}\times G}R(\omega)d\nu_{\omega,\eta,g}\\=&\rho^{-1}\cdot\sum_{\alpha\in\mathcal{W}^{1}(\omega) \atop \alpha\eta(\vartheta\omega) \ \omega\text{\rm -admissible}}\exp\left(\varphi(\omega,\alpha\eta(\vartheta\omega))\right)\int_{\{\omega\}\times\mathcal{E}_{\omega}\times G}R(\omega)d\nu_{\omega,\alpha\eta(\vartheta\omega),g\cdot (\psi_{\omega}(\alpha))^{-1}}.
\end{align*}
\end{proof}

In the following, we give the proof of Theorem \ref{the third result of main theorems}.

\begin{proof}[Proof of Theorem \ref{the third result of main theorems}]
Since $G$ is an amenable group, let $\mathfrak{M}$ be the full Banach mean value for $G$ satisfying for any function $\phi:G\rightarrow\mathbb{R}$, one has
\begin{align*}
\inf_{x\in G}\phi(x)\leq \mathfrak{M}(\phi)\leq \sup_{x\in G}\phi(x).
\end{align*}
By the definitions of $\Pi_{Gur}^{(r)}(\mathcal{T}_{\rm ab},\tilde{\varphi})$ and $\Pi_{Gur}^{(r)}(\mathcal{T},\tilde{\varphi})$, it can be proved that
\begin{align*}
\Pi_{Gur}^{(r)}(\mathcal{T}_{\rm ab},\tilde{\varphi})\geq\Pi_{Gur}^{(r)}(\mathcal{T},\tilde{\varphi}).
\end{align*}
Next, we will prove
\begin{align*}
\Pi_{Gur}^{(r)}(\mathcal{T}_{\rm ab},\tilde{\varphi})\leq\Pi_{Gur}^{(r)}(\mathcal{T},\tilde{\varphi}).
\end{align*}
Fix a measurable family $\eta$, and non-negative random continuous function $R:\mathcal{E}\times G\rightarrow\mathcal{R}$. By Lemma \ref{the lemma 3 of main result} and (3) in Lemma \ref{the lemma 1 of main result}, we have
\begin{align*}
&\rho^{n}\cdot\int_{\{\omega\}\times\mathcal{E}_{\omega}\times G}R(\omega)d\nu_{\omega,\eta,g}\\=&\sum_{\alpha\in\mathcal{W}^{n}(\omega) \atop \alpha\eta(\vartheta^{n}\omega) \ \omega\text{\rm -admissible}}\exp\left(S_{n}\varphi(\omega,\alpha\eta(\vartheta^{n}\omega))\right)\int_{\{\omega\}\times\mathcal{E}_{\omega}\times G}R(\omega)d\nu_{\omega,\alpha\eta(\vartheta^{n}\omega),g\cdot (\psi_{\omega}^{n}(\alpha))^{-1}}\\ \geq& C_{3}(\omega)^{-1}\cdot \sum_{\alpha\in\mathcal{W}^{n}(\omega) \atop \alpha\eta(\vartheta^{n}\omega) \ \omega\text{\rm -admissible}}\exp\left(S_{n}\varphi(\omega,\alpha\eta(\vartheta^{n}\omega))\right)\int_{\{\omega\}\times\mathcal{E}_{\omega}\times G}R(\omega)d\nu_{\omega,\eta,g\cdot (\psi_{\omega}^{n}(\alpha))^{-1}}.
\end{align*}
Then by the above inequality and the linear property of $\mathfrak{M}$, we have
\begin{align*}
&C_{3}(\omega)\cdot\rho^{n}\cdot\mathfrak{M}\left[g\mapsto\frac{\int_{\{\omega\}\times\mathcal{E}_{\omega}\times G}R(\omega)d\nu_{\omega,\eta,g}}{\int_{\{\omega\}\times\mathcal{E}_{\omega}\times G}R(\omega)d\nu_{\omega,\eta,g}}\right]
\\ \geq&\mathfrak{M}\left[g\mapsto\sum_{\alpha\in\mathcal{W}^{n}(\omega) \atop \alpha\eta(\vartheta^{n}\omega) \ \omega\text{\rm -admissible}}\exp\left(S_{n}\varphi(\omega,\alpha\eta(\vartheta^{n}\omega))\right)\frac{\int_{\{\omega\}\times\mathcal{E}_{\omega}\times G}R(\omega)d\nu_{\omega,\eta,g\cdot (\psi_{\omega}^{n}(\alpha))^{-1}}}{\int_{\{\omega\}\times\mathcal{E}_{\omega}\times G}R(\omega)d\nu_{\omega,\eta,g}}\right]\\
\geq&\sum_{\alpha\in\mathcal{W}^{n}(\omega) \atop {\alpha\eta(\vartheta^{n}\omega) \atop \omega\text{\rm -admissible}}}\exp\left(S_{n}\varphi(\omega,\alpha\eta(\vartheta^{n}\omega))+\mathfrak{M}\left[g\mapsto\log\frac{\int_{\{\omega\}\times\mathcal{E}_{\omega}\times G}R(\omega)d\nu_{\omega,\eta,g\cdot (\psi_{\omega}^{n}(\alpha))^{-1}}}{\int_{\{\omega\}\times\mathcal{E}_{\omega}\times G}R(\omega)d\nu_{\omega,\eta,g}}\right]\right).
\end{align*}
Define the map $\Phi:\Omega\times G,(\omega,h)\mapsto \Phi_{\omega}(h)$ by a function
\begin{align*}
\Phi_{\omega}:G\rightarrow\mathbb{R}, \ \ h\mapsto \mathfrak{M}\left[g\mapsto\log\frac{\int_{\{\omega\}\times\mathcal{E}_{\omega}\times G}R(\omega)d\nu_{\omega,\eta,g\cdot h}}{\int_{\{\omega\}\times\mathcal{E}_{\omega}\times G}R(\omega)d\nu_{\omega,\eta,g}}\right].
\end{align*}

{\bf Claim.} The function $\Phi_{\omega}:G\rightarrow\mathbb{R}$ is a homomorphism for $\mathbb{P}$-a.e. $\omega\in \Omega$ of which that satisfies the following two conditions:
\begin{itemize}
  \item[(1)] $\Phi_{\omega}(h_{1}\cdot h_{2})=\Phi_{\omega}(h_{1})+\Phi_{\omega}(h_{2})$;
  \item[(2)] $\Phi_{\omega}(id)=0$.
\end{itemize}

\noindent{\bf Proof of Claim.} Consider the equation
\begin{align*}
\frac{\int_{\{\omega\}\times\mathcal{E}_{\omega}\times G}R(\omega)d\nu_{\omega,\eta,g\cdot h_{2} h_{1}}}{\int_{\{\omega\}\times\mathcal{E}_{\omega}\times G}R(\omega)d\nu_{\omega,\eta,g}}=\frac{\int_{\{\omega\}\times\mathcal{E}_{\omega}\times G}R(\omega)d\nu_{\omega,\eta,g\cdot h_{2} h_{1}}}{\int_{\{\omega\}\times\mathcal{E}_{\omega}\times G}R(\omega)d\nu_{\omega,\eta,g\cdot h_{2}}}\frac{\int_{\{\omega\}\times\mathcal{E}_{\omega}\times G}R(\omega)d\nu_{\omega,\eta,g\cdot h_{2}}}{\int_{\{\omega\}\times\mathcal{E}_{\omega}\times G}R(\omega)d\nu_{\omega,\eta,g}}.
\end{align*}
Since $\mathfrak{M}$ is the Banach Mean for $G$, then we have
\begin{align*}
&\Phi_{\omega}(h_{1}\cdot h_{2})\\=&\mathfrak{M}\left[g\mapsto\log\frac{\int_{\{\omega\}\times\mathcal{E}_{\omega}\times G}R(\omega)d\nu_{\omega,\eta,g\cdot h_{2} h_{1}}}{\int_{\{\omega\}\times\mathcal{E}_{\omega}\times G}R(\omega)d\nu_{\omega,\eta,g}}\right]\\
=&\mathfrak{M}\left[g\mapsto\log\frac{\int_{\{\omega\}\times\mathcal{E}_{\omega}\times G}R(\omega)d\nu_{\omega,\eta,g\cdot h_{2} h_{1}}}{\int_{\{\omega\}\times\mathcal{E}_{\omega}\times G}R(\omega)d\nu_{\omega,\eta,g\cdot h_{2}}}\right]+\mathfrak{M}\left[g\mapsto\log\frac{\int_{\{\omega\}\times\mathcal{E}_{\omega}\times G}R(\omega)d\nu_{\omega,\eta,g\cdot h_{2}}}{\int_{\{\omega\}\times\mathcal{E}_{\omega}\times G}R(\omega)d\nu_{\omega,\eta,g}}\right]\\
=&\mathfrak{M}\left[g\mapsto\log\frac{\int_{\{\omega\}\times\mathcal{E}_{\omega}\times G}R(\omega)d\nu_{\omega,\eta,g\cdot h_{1}}}{\int_{\{\omega\}\times\mathcal{E}_{\omega}\times G}R(\omega)d\nu_{\omega,\eta,g}}\right]+\mathfrak{M}\left[g\mapsto\log\frac{\int_{\{\omega\}\times\mathcal{E}_{\omega}\times G}R(\omega)d\nu_{\omega,\eta,g\cdot h_{2}}}{\int_{\{\omega\}\times\mathcal{E}_{\omega}\times G}R(\omega)d\nu_{\omega,\eta,g}}\right]\\
=&\Phi_{\omega}(h_{1})+\Phi_{\omega}(h_{2}).
\end{align*}
This obtains the condition (1). The condition (2) can also be proved by the definition of $\Phi_{\omega}$ and $\mathfrak{M}$.

Let $\alpha\in\mathcal{W}^{n}_{a,a}(\omega)$ and $\psi_{{\rm ab},\omega}^{n}(\alpha)=0$ where $a\in\mathcal{W}^{1}$ and $\vartheta^{n}\omega\in\Omega_{a}$. Then by the definition of $\psi_{{\rm ab},\omega}^{n}$, we have $\pi(\psi_{\omega}^{n}(\alpha))=0$ and thus, $\psi_{\omega}^{n}(\alpha)\in [G,G]$. Then by the definition of $[G,G]$, there exists $m\in\mathbb{N}$ such that
\begin{align*}
\Phi_{\omega}(\psi_{\omega}^{n}(\alpha))=\Phi_{\omega}\left(\prod_{i=0}^{m-1}g_{i}h_{i}g_{i}^{-1}h_{i}^{-1}\right)=0
\end{align*}
where $g_{i},h_{i}\in G$ for any $0\leq i\leq m-1$. By the topologically mixing property, there exists $N\geq 1$ such that for any $\alpha\in\mathcal{W}^{n}_{a,a}(\omega)$, there exists a $\vartheta^{n}\omega$-admissible word $\alpha^{*}$ of length $N+1$ such that $\alpha\alpha^{*}\in\mathcal{W}^{n+N+1}(\omega)$.
Then we have
\begin{align*}
\mathcal{Z}_{{\rm ab},n}^{\omega}(\tilde{\varphi}_{\rm ab},a,a)&=\sum_{\alpha\in\mathcal{W}^{n}_{a,a}(\omega), \ \psi_{{\rm ab},\omega}^{n}(\alpha)=0}\exp\left(\sum_{i=0}^{n-1}\tilde{\varphi}_{\rm ab}\circ(\Theta_{\psi_{\rm ab}})^{i}(\omega,\alpha o(\vartheta^{n}\omega),id)\right)\\&=\sum_{\alpha\in\mathcal{W}^{n}_{a,a}(\omega), \ \psi_{{\rm ab},\omega}^{n}(\alpha)=0}\exp\left(S_{n}\varphi(\omega,\alpha o(\vartheta^{n}\omega))\right)\cdot\exp\left(\Phi_{\omega}((\psi_{\omega}^{n}(\alpha))^{-1})\right)\\
&\leq C_{4}(\omega)\cdot\sum_{\alpha\in\mathcal{W}^{n+N+1}(\omega) \atop  {\alpha\eta(\vartheta^{n+N+1}\omega) \atop \omega\text{\rm -admissible}}}\exp\left(S_{n+N+1}\varphi(\omega,\alpha\eta(\vartheta^{n+N+1}\omega))\right)\cdot\Delta(\omega)\\
&\leq C_{4}(\omega)\cdot C_{3}(\omega)\cdot\rho^{n+N+1},
\end{align*}
where
\begin{align*}
\Delta(\omega):=\exp\left(\Phi_{\omega}((\psi_{\omega}^{n+N+1}(\alpha))^{-1})\right)
\end{align*}
and
\begin{align*}
C_{4}(\omega):=C_{2}(\omega)^{-(N+1)}\cdot (B_{\varphi})^{-(N+1)}\cdot \left(\inf_{x\in\mathbb{N}^{\mathbb{N}}}\exp(\Phi_{\omega}\left(\psi(x))\right)\right)^{-(N+1)}.
\end{align*}
It follows that
\begin{align*}
\Pi_{Gur}^{(r)}(\mathcal{T}_{\rm ab},\tilde{\varphi}_{\rm ab})&=\int_{\Omega}\lim_{n\rightarrow\infty, \ n\in J_{\omega}(\Omega_{a})}\frac{1}{n}\log \tilde{Z}_{{\rm ab},n}^{\omega}(\tilde{\varphi}_{\rm ab},a,a)d\mathbb{P}(\omega)\\&\leq\int_{\Omega}\lim_{n\rightarrow\infty, \ n\in J_{\omega}(\Omega_{a})}\frac{1}{n}\log \left(C_{4}(\omega)\cdot C_{3}(\omega)\cdot\rho^{n+N+1}\right)d\mathbb{P}(\omega)\\
&=\int_{\Omega}\lim_{n\rightarrow\infty, \ n\in J_{\omega}(\Omega_{a})}\frac{1}{n}\log \left(\rho^{n+N+1}\right)d\mathbb{P}(\omega)\\
&\leq\log \rho.
\end{align*}
Since $\rho=\exp\left(\Pi_{Gur}^{(r)}(\mathcal{T},\tilde{\varphi})\right)$, then we have
\begin{align*}
\Pi_{Gur}^{(r)}(\mathcal{T}_{\rm ab},\tilde{\varphi}_{\rm ab})\leq\Pi_{Gur}^{(r)}(\mathcal{T},\tilde{\varphi}).
\end{align*}
\end{proof}

\section{Relative Gurevi\v{c} entropy of random group extensions}\label{Relative Gurevic entropy of random group extensions}

In this section, we study the relative Gurevi\v{c} entropies of random group extensions and give the proof of Theorem \ref{the fourth result of main theorems}. Firstly, we give the following result by considering the non-amenable case. For a general case, by the definitions of $h_{Gur}^{(r)}(\mathcal{T}_{\rm ab})$ and $h_{Gur}^{(r)}(\mathcal{T})$, we have the following inequality.
\begin{align*}
h_{Gur}^{(r)}(\mathcal{T})\leq h_{Gur}^{(r)}(\mathcal{T}_{\rm ab}).
\end{align*}

\begin{theorem}\label{the one hand of the fourth result of main theorems}
Let $\mathcal{T}$ be a topologically mixing random group extension of a random shift of finite type $f$ by a countable group $G$. If $G$ is non-amenable, then
$$h_{Gur}^{(r)}(\mathcal{T})<h_{Gur}^{(r)}(\mathcal{T}_{\rm ab}).$$
\end{theorem}

\begin{proof}
Since $G^{\rm ab}=G/[G,G]$, then we have
\begin{align*}
G^{\rm ab}=\mathbb{Z}^{a}\times G_{0}
\end{align*}
for some $a\geq 0$ and finite abelian group $G_{0}$.

{\bf Case 1.} Suppose that $a>0$.
The random group extension $\mathcal{T}_{{\rm ab}}:=(\mathcal{T}_{{\rm ab}, \omega})_{\omega\in\Omega}$ where
\begin{align*}
\mathcal{T}_{{\rm ab}, \omega}:\mathcal{E}_{\omega}\times G^{{\rm ab}}\rightarrow\mathcal{E}_{\vartheta\omega}\times G^{{\rm ab}}, \ (x,g)\mapsto (f_{\omega}x,g\cdot\psi_{{\rm ab},\omega}(x)), \ \omega\in \Omega,
\end{align*}
that $\psi_{{\rm ab},\omega}=\pi\circ\psi_{\omega}$ and $\pi:G\rightarrow G^{ab}$ is the natural projection induces a random group extension on fibers $$\mathcal{E}_{\omega}\times \mathbb{Z}^{a}\rightarrow\mathcal{E}_{\vartheta\omega}\times \mathbb{Z}^{a}, \ \omega\in \Omega$$ which is similarly denoted by $\mathcal{T}_{{\rm ab}}$. By following the processes of \cite{Pollicott1994Rates}, we observe that there exists $\xi\in\mathbb{R}^{a}$ such that
\begin{align}\label{extension of relative topological pressures}
h_{Gur}^{(r)}(\mathcal{T}_{\rm ab})=\Pi_{top}^{(r)}(f,\varphi_{{\rm ab}}),
\end{align}
where $\varphi_{{\rm ab}}$ is defined by $$\varphi_{{\rm ab}}(\omega,x):=\langle\xi,  \psi_{{\rm ab}}\rangle(\omega,x)=\sum_{i=1}^{a}\xi_{i}\cdot \psi_{{\rm ab},\omega}^{(i)}(x)$$ and   $$\xi=(\xi_{1},\xi_{2},\ldots,\xi_{a})\in\mathbb{R}^{a} \ \ \text{and} \ \ \psi_{{\rm ab},\omega}=(\psi_{{\rm ab},\omega}^{(1)},\psi_{{\rm ab},\omega}^{(2)},\ldots,\psi_{{\rm ab},\omega}^{(a)})$$ which $\psi_{{\rm ab},\omega}^{(i)}$ is a locally fiber H\"{o}lder continuous function. Indeed, on the one hand, combining with the processes of \cite{Parry1990Zeta} and \cite{Kifer2008Thermodynamic}, there exists a unique $\xi\in\mathbb{R}^{a}$ such that
\begin{align*}
\int\psi_{{\rm ab}}d\mu_{\langle\xi,  \psi_{{\rm ab}}\rangle}:=\left(\int\psi_{{\rm ab}}^{(1)}d\mu_{\langle\xi,  \psi_{{\rm ab}}\rangle},\int\psi_{{\rm ab}}^{(2)}d\mu_{\langle\xi,  \psi_{{\rm ab}}\rangle},\ldots,\int\psi_{{\rm ab}}^{(a)}d\mu_{\langle\xi,  \psi_{{\rm ab}}\rangle}\right)={\bf 0}
\end{align*}
where $\mu_{\langle\xi,  \psi_{{\rm ab}}\rangle}$ is the
equilibrium state of $\langle\xi,  \psi_{{\rm ab}}\rangle$ and $$\Pi_{top}^{(r)}(f,\langle\xi,  \psi_{{\rm ab}}\rangle)=\min_{\eta\in \mathbb{R}^{a}}\Pi_{top}^{(r)}(f,\langle\eta,  \psi_{{\rm ab}}\rangle).$$ Then if $\mu\neq \mu_{\langle\xi,  \psi_{{\rm ab}}\rangle}$ such that $\int\psi_{{\rm ab}}d\mu_{\langle\xi,  \psi_{{\rm ab}}\rangle}={\bf 0}$, we have
\begin{align*}
\Pi_{top}^{(r)}(f,\langle\xi,  \psi_{{\rm ab}}\rangle)&=h_{\mu_{\langle\xi,  \psi_{{\rm ab}}\rangle}}^{(r)}(f)+\int\langle\xi,  \psi_{{\rm ab}}\rangle d\mu_{\langle\xi,  \psi_{{\rm ab}}\rangle}\\&\leq h_{\mu}^{(r)}(f),
\end{align*}
Then we have
\begin{align*}
\Pi_{top}^{(r)}(f,\langle\xi,  \psi_{{\rm ab}}\rangle)\leq \sup_{\mu\in\mathcal{M}_{\mathbb{P}}^{1}(\mathcal{E},f)}\left\{h_{\mu}^{(r)}(f):\int\psi_{{\rm ab}}d\mu={\bf 0}\right\}.
\end{align*}
The other inequality can be obtained by the definitions. Then we have
\begin{align*}
\Pi_{top}^{(r)}(f,\langle\xi,  \psi_{{\rm ab}}\rangle)= \sup_{\mu\in\mathcal{M}_{\mathbb{P}}^{1}(\mathcal{E},f)}\left\{h_{\mu}^{(r)}(f):\int\psi_{{\rm ab}}d\mu={\bf 0}\right\}.
\end{align*}
On the other hand, by the processes of \cite{Marcus1990Entropy} and \cite{Pollicott1994Rates}, we have
\begin{align*}
\tilde{Z}_{{\rm ab},n_{a}^{(j)}(\omega)}^{\omega}(\tilde{0}_{\rm ab},a,a)\leq n_{a}^{(j)}(\omega)\left(n_{a}^{(j+1)}(\omega)\right)^{r(\omega)+1}\exp\left(n\left(\Pi_{top}^{(r)}(f,\langle\xi,  \psi_{{\rm ab}}\rangle)\right)\right)
\end{align*}
where $$r(\omega)=\sum_{i\in S(\omega),j\in S(\vartheta\omega)}\alpha_{i,j}(\omega)$$ and for any $\epsilon>0$, there exists $m,N\in\mathbb{N}$ and $\delta>0$ such that
\begin{align*}
\tilde{Z}_{{\rm ab},n_{a}^{(j)}(\omega)}^{\omega}(\tilde{0}_{\rm ab},a,a)\geq\delta \left(n_{a}^{(j)}(\omega)\right)^{-m+1}\exp\left(n\left(\Pi_{top}^{(r)}(f,\langle\xi,  \psi_{{\rm ab}}\rangle)-\epsilon\right)\right)
\end{align*}
for any $j\geq N$. Combining with the above inequalities, it follows that
\begin{align*}
h_{Gur}^{(r)}(\mathcal{T}_{\rm ab})&=\int_{\Omega}\lim_{j\rightarrow\infty}\frac{1}{n_{a}^{(j)}(\omega)}\log \tilde{Z}_{{\rm ab},n_{a}^{(j)}(\omega)}^{\omega}(\tilde{0}_{\rm ab},a,a)d\mathbb{P}(\omega)\\&=\Pi_{top}^{(r)}(f,\langle\xi,  \psi_{{\rm ab}}\rangle)
\end{align*}
and
\begin{align}\label{variational principle of random group extensions}
h_{Gur}^{(r)}(\mathcal{T}_{\rm ab})=\sup_{\mu\in\mathcal{M}_{\mathbb{P}}^{1}(\mathcal{E},f)}\left\{h_{\mu}^{(r)}(f):\int\psi_{{\rm ab}}d\mu={\bf 0}\right\}.
\end{align}
This implies the \eqref{extension of relative topological pressures}.
Since $G$ is not amenable, by Remark \ref{the main remark of main results} we can obtain that
\begin{align}\label{inequality 1 the fourth result of main theorems}
\begin{split}
\log {\rm spr}_{\mathcal{H}}(\mathcal{L}_{\tilde{\varphi}_{{\rm ab}}}^{\omega})&<\Pi_{Gur}^{(r)}(f,\varphi_{{\rm ab}})=\Pi_{top}^{(r)}(f,\varphi_{{\rm ab}})\\&=h_{Gur}^{(r)}(\mathcal{T}_{\rm ab}).
\end{split}
\end{align}
Combining with (3) in Lemma \ref{the description of spectral of random group extension} and Proposition \ref{the estimation 1 of random operators}, it can be proved that
\begin{align*}
{\rm spr}_{\mathcal{H}}(\mathcal{L}_{\tilde{\varphi}_{{\rm ab}}}^{\omega})&=\limsup\limits_{n\rightarrow\infty}\mathcal{A}_{n}(\omega)^{\frac{1}{n}}\\&\geq \exp\left(\Pi_{Gur}^{(r)}(\mathcal{T},\tilde{\varphi}_{{\rm ab}})\right)
\end{align*}
for $\mathbb{P}$-a.e. $\omega\in \Omega$.
Then for $\mathbb{P}$-a.e. $\omega\in \Omega$, we have
\begin{align}\label{inequality 2 the fourth result of main theorems}
\log{\rm spr}_{\mathcal{H}}(\mathcal{L}_{\tilde{\varphi}_{{\rm ab}}}^{\omega})\geq\Pi_{Gur}^{(r)}(\mathcal{T},\tilde{\varphi}_{{\rm ab}}).
\end{align}
By \eqref{inequality 1 the fourth result of main theorems} and \eqref{inequality 2 the fourth result of main theorems}, we have
\begin{align*}
\Pi_{Gur}^{(r)}(\mathcal{T},\tilde{\varphi}_{{\rm ab}})<h_{Gur}^{(r)}(\mathcal{T}_{\rm ab}).
\end{align*}
For $\alpha\in\mathcal{W}^{n}(\omega)$, since $\psi_{\omega}^{n}(\alpha)=id$ implies that $\psi_{{\rm ab},  \omega}^{n}(\alpha)=0$, then
\begin{align*}
\Pi_{Gur}^{(r)}(\mathcal{T},\tilde{\varphi}_{{\rm ab}})&=\int_{\Omega}\lim_{j\rightarrow\infty}\frac{1}{n_{a}^{(j)}(\omega)}\log \tilde{Z}_{n_{a}^{(j)}(\omega)}^{\omega}(\tilde{\varphi}_{\rm ab},a,a)d\mathbb{P}(\omega)\\&=h_{Gur}^{(r)}(\mathcal{T}).
\end{align*}
Then this implies the result.

{\bf Case 2.} If $a=0$, then $G^{\rm ab}$ is finite. Then we have $$h_{Gur}^{(r)}(\mathcal{T}_{\rm ab})=h_{top}^{(r)}(f).$$
It follows that
\begin{align*}
h_{Gur}^{(r)}(\mathcal{T})&\leq \log {\rm spr}_{\mathcal{H}}(\mathcal{L}_{\tilde{0}}^{\omega})\\&<\Pi_{Gur}^{(r)}(f,0)\\&=\Pi_{top}^{(r)}(f,0)=h_{top}^{(r)}(f)
\end{align*}
by using the above similar processes in Case 1,
then
\begin{align*}
h_{Gur}^{(r)}(\mathcal{T})\leq h_{Gur}^{(r)}(\mathcal{T}_{\rm ab})
\end{align*}
which implies the proof.

\end{proof}

Next, we prove the Theorem \ref{the fourth result of main theorems}.

\begin{proof}[Proof of Theorem \ref{the fourth result of main theorems}]
By the definitions of $h_{Gur}^{(r)}(\mathcal{T}_{\rm ab})$ and $h_{Gur}^{(r)}(\mathcal{T})$, it can be proved that
\begin{align*}
h_{Gur}^{(r)}(\mathcal{T})\leq h_{Gur}^{(r)}(\mathcal{T}_{\rm ab}).
\end{align*}
On the one hand, suppose that $G$ is amenable, by using Theorem \ref{the third result of main theorems} of which $\varphi=0$, we have
$$h_{Gur}^{(r)}(\mathcal{T})= h_{Gur}^{(r)}(\mathcal{T}_{\rm ab}).$$
On the another hand, assume that $$h_{Gur}^{(r)}(\mathcal{T})= h_{Gur}^{(r)}(\mathcal{T}_{\rm ab}),$$ if $G$ is non-amenable,
then it can be proved that $$h_{Gur}^{(r)}(\mathcal{T})< h_{Gur}^{(r)}(\mathcal{T}_{\rm ab})$$
by using Theorem \ref{the one hand of the fourth result of main theorems}, which is a contradiction.
\end{proof}

\section*{Acknowledgements}

The work was supported by the
National Natural Science Foundation of China (Nos. 12071222 and 11971236), China Postdoctoral Science Foundation (No.2016M591873),
and China Postdoctoral Science Special Foundation (No.2017T100384). The first author was supported by Project funded by China Postdoctoral Science Foundation (No.2023TQ0066) and supported by the Postdoctoral Fellowship Program of China Postdoctoral Science Foundation (No.GZC20230536). The work was also funded by the Priority Academic Program Development of Jiangsu Higher Education Institutions.  We would like to express our gratitude to Tianyuan Mathematical Center in Southwest China (11826102), Sichuan University and Southwest Jiaotong University for their support and hospitality.

\section*{Data Availability}

Data sharing is not applicable to this article as no datasets were generated or analyzed during the current study.


\begin{thebibliography}{99}

\bibitem{Arnold1998Random}
        {L. Arnold}, {\em Random dynamical systems}, Springer-Verlag, 1998.

\bibitem{Bogenschutz1995Ruelle}
        {T. Bogensch\"{u}tz and V.M. Gundlach}, {Ruelle's transfer operator for random subshifts of finite
type}, \emph{Ergodic Theory Dynam. Systems}, {\bf 15} (1995), 413-447.

\bibitem{Brooks1985The}
        {R. Brooks}, {The bottom of the spectrum of a Riemannian covering}, \emph{J. Reine Angew. Math.} {\bf 357} (1985), 101-114.

\bibitem{Cornfeld1982Ergodic}
        {I.P. Cornfeld, S.V. Fomin and Y.G. Sina\u{i}}, {\em Ergodic Theory}, Grundlehren der Mathematischen Wissenschaften vol. 245 [Fundamental Principles of Mathematical Sciences]. Springer-Verlag, New York, 1982.

\bibitem{Day1964Convolutions}
        {M.M. Day}, {Convolutions, means, and spectra}, \emph{Ill. J. Math.} {\bf 8} (1964), 100-111.

\bibitem{Denker1991On}
        {M. Denker and M. Urbanski}, {On the existence of conformal measures}, \emph{Trans. Am. Math. Soc.}
{\bf 328} (1991), 563-587.

\bibitem{Denker2008Thermodynamic}
        {M. Denker, Y. Kifer and M. Stadlbauer}, {Thermodynamic formalism for random countable Markov shifts},
\emph{Discrete Contin. Dyn. Syst.} {\bf 22} (2008), 131-164.

\bibitem{Dooley2015Local}
        {A. Dooley and G. Zhang}, {Local entropy theory of a random dynamical system}, \emph{Mem. Amer. Math. Soc.} {\bf 233} (2015), vi+106 pp.

\bibitem{Dougall2016Amenability}
        {R. Dougall and R. Sharp}, {Amenability, critical exponents of subgroups and growth of closed
geodesics}, \emph{Math. Ann.} {\bf 365} (2016), 1359-1377.

\bibitem{Dougall2019Critical}
        {R. Dougall}, {Critical exponents of normal subgroups, the spectrum of group extended
transfer operators, and Kazhdan distance}, \emph{Adv. Math.} {\bf 349} (2019), 316-347.

\bibitem{Dougall2021Anosov}
        {R. Dougall and R. Sharp}, {Anosov flows, growth rates on covers and group extensions of subshifts}, \emph{Invent. Math.} {\bf 223} (2021), 445-483.

\bibitem{Downarowicz2023Symbolic}
        {T. Downarowicz and G. Zhang}, {Symbolic extensions of amenable group actions and the comparison property}, \emph{Mem. Amer. Math. Soc.} {\bf 281} (2023), vi+95 pp.

\bibitem{Downarowicz2023Multiorders}
        {T. Downarowicz, P. Oprocha, W. Mateusz and G. Zhang}, {Multiorders in amenable group actions}, \emph{Groups Geom. Dyn.} to appear.

\bibitem{Folner1954Generalization}
        {E. F{\o}lner}, {Generalization of a theorem of Bogoliouboff to topological abelian groups}. With an appendix on Banach mean values in non-abelian groups, \emph{Math. Scand.} {\bf 2} (1954), 5-18.

\bibitem{Grigorchuk1980Symmetrical}
        {R.I. Grigorchuk}, {Symmetrical random walks on discrete groups}, in: Multicomponent Random Systems, in: Adv.
Probab. Related Topics, vol. 6, Dekker, New York, 1980, 285-325.

\bibitem{Folner1955On}
        {E. F{\o}lner}, {On groups with full banach mean value}, \emph{Math. Scand.} {\bf 3} (1955), 243-254.

\bibitem{Gurevic1969Topological}
        {B.M. Gurevi\v{c}}, {Topological entropy for denumerable Markov chains}, \emph{Soviet Math. Dokl.} {\bf 10} (1969),
911-915.

\bibitem{Gurevic1970Shift}
        {B.M. Gurevi\v{c}}, {Shift entropy and Markov measures in the space of paths of a countable graph}, (Russian) \emph{Dokl. Akad. Nauk SSSR} {\bf 192} (1970), 963-965.

\bibitem{Huang2017Entropy}
        {W. Huang and K. Lu}, {Entropy, chaos, and weak horseshoe for infinite-dimensional random dynamical systems}, \emph{Comm. Pure Appl. Math.} {\bf 70} (2017), 1987-2036.

\bibitem{Huang2019Ergodic}
        {W. Huang, Z. Lian and K. Lu}, {Ergodic theory of random Anosov systems mixing on fibers}, \emph{arXiv preprint, arXiv:
1612.08394}, 2019.

\bibitem{Jaerisch2014Fractal}
        {J. Jaerisch}, {Fractal models for normal subgroups of Schottky groups}, \emph{Trans. Amer. Math. Soc.} {\bf 366} (2014), 5453-5485.

\bibitem{Jaerisch2015Groupextended}
        {J. Jaerisch}, {Group-extended Markov systems, amenability, and the Perron-Frobenius operator}, \emph{Proc. Am. Math. Soc.} {\bf 143} (2015), 289-300.

\bibitem{Jaerisch2016Recurrence}
        {J. Jaerisch}, {Recurrence and pressure for group extensions}, \emph{Ergodic Theory Dynam. Systems} {\bf 36} (2016), 108-126.

\bibitem{Kaimanovich1983Random}
        {V.A. Kaimanovich and A.M. Vershik}, {Random walks on discrete groups: boundary and entropy}, \emph{Ann. Probab.} {\bf 11} (1983), 457-490.

\bibitem{Kesten1959Full}
        {H. Kesten}, {Full Banach mean values on countable groups}, \emph{Math. Scand.} {\bf 7} (1959) 146-156.

\bibitem{Khanin1996Thermodynamic}
        {K. Khanin and Y. Kifer}, {\em Thermodynamic formalism for random transformations and statistical mechanics}, Amer. Math. Soc. Transl. Ser. 2 {\bf 171} (1996), 107-140.

\bibitem{Kifer1986Ergodic}
        {Y. Kifer}, {\em Ergodic Theory of Random Transformations}, Birkh\"{a}user, 1986.

\bibitem{Kifer1992Equilibrium}
        {Y. Kifer}, {Equilibrium states for random expanding transformations}, \emph{Random Comput. Dynam.} {\bf 1} (1992/93), 1-31.

\bibitem{Kifer1995Multidimensional}
        {Y. Kifer}, {Multidimensional random subshifts of finite type and their large deviations}, \emph{Probab. Theory Related Fields} {\bf 102} (1995), 223-248.

\bibitem{Kifer2001On}
        {Y. Kifer}, {On the topological pressure for random bundle transformations}, \emph{Amer. Math. Soc. Transl.} {\bf 202} (2001), 197-214.

\bibitem{Kifer2006Random}
        {Y. Kifer and P. Liu}, {\em Random dynamics}, Handbook of Dynamical Systems, eds. B. Hasselblatt and A. Katok (Elsevier, 2006), 379-499.

\bibitem{Kifer2008Thermodynamic}
        {Y. Kifer}, {\em Thermodynamic formalism for random transformations revisited}, Stoch. Dyn. {\bf 8} (2008), 77-102.

\bibitem{Lindenstrauss2001Pointwise}
        {E. Lindenstrauss}, {Pointwise theorems for amenable groups}, \emph{Invent. Math.} {\bf 146} (2001), 259-295.

\bibitem{Marcus1990Entropy}
        {B. Marcus and S. Tuncel}, {Entropy at a weight-per-symbol and embeddings of Markov chains},
\emph{Invent. Math.} {\bf 102} (1990), 235-266.

\bibitem{Parry1990Zeta}
        {W. Parry and M. Pollicott}, {Zeta functions and the periodic orbit structure of hyperbolic
dynamics}, \emph{Asterisque}, {\bf 187-188} (1990), 1-268.

\bibitem{Pollicott1994Rates}
        {M. Pollicott and R. Sharp}, {Rates of recurrence for $\mathbb{Z}^{q}$ and $\mathbb{R}^{q}$ extensions of subshifts of finite
type}, \emph{J. London. Math. Soc.} {\bf 49} (1994), 401-416.

\bibitem{Roblin2005A}
        {T. Roblin}, {A Fatou theorem for conformal densities with applications to Galois coverings in negative curvature (Un
th\'{e}or\`{e}me de Fatou pour les densit\'{e}s conformes avec applications aux rev\^{e}tements galoisiens en courbure n\'{e}gative)},
\emph{Israel J. Math.} {\bf 147} (2005), 333-357.

\bibitem{Sarig1999Thermodynamic}
        {O.M. Sarig}, {Thermodynamic formalism for countable Markov shifts}, \emph{Ergodic Theory Dynam. Systems} {\bf 19}
(1999), 1565-1593.

\bibitem{Sarig2001Thermodynamic}
        {O.M. Sarig}, {Thermodynamic formalism for null recurrent potentials}, \emph{Israel J. Math.} {\bf 121}
(2001), 285-311.

\bibitem{Sarig2003Existence}
        {O.M. Sarig}, {Existence of Gibbs measures for countable Markov shifts}, \emph{Proc. Amer. Math. Soc.} {\bf 131} (2003)
1751-1758.

\bibitem{Sarig2013Symbolic}
        {O.M. Sarig}, {Symbolic dynamics for surface diffeomorphisms with positive entropy}, \emph{J. Amer. Math. Soc.} {\bf 26} (2013), 341-426.

\bibitem{Sharp2007Critical}
        {R. Sharp}, {Critical exponents for groups of isometries}, \emph{Geom. Dedicata} {\bf 125} (2007), 63-74.

\bibitem{Stadlbauer2010On}
        {M. Stadlbauer}, {On random topological Markov chains with big images and preimages}, \emph{Stoch. Dyn.} {\bf 10}
(2010), 77-95.

\bibitem{Stadlbauer2013An}
        {M. Stadlbauer}, {An extension of Kesten's criterion for amenability to topological Markov chains}, \emph{Adv. Math.} {\bf 235} (2013), 450-468.

\bibitem{Stadlbauer2017Coupling}
        {M. Stadlbauer}, {Coupling methods for random topological Markov chains}, \emph{Ergodic Theory Dynam. Systems} {\bf 37} (2017), 971-994.

\bibitem{Stadlbauer2021Thermodynamic}
        {M. Stadlbauer, S. Suzuki and P. Varandas}, {Thermodynamic formalism for random non-uniformly expanding
maps}, \emph{Comm. Math. Phys.} {\bf 385} (2021), 369-427.

\end{thebibliography}
\end{document}